\documentclass[final,3p,times]{elsarticle}

\usepackage{xcolor}
\definecolor{reviewerone}{RGB}{0,0,0}
\definecolor{reviewertwo}{RGB}{0,0,0}

\usepackage{amssymb}
\usepackage{amsmath}
\usepackage{amsthm}
\usepackage[colorlinks=true,linkcolor=blue,citecolor=blue]{hyperref}
\usepackage{algorithm}
\usepackage{algpseudocode}
\usepackage{booktabs} 
\usepackage{multirow} 
\usepackage{float}
\usepackage{subcaption}
\usepackage[capitalize]{cleveref}
\newtheorem{remark}{Remark}

\newtheorem{theorem}{Theorem}

\newtheorem{assumption}[theorem]{Assumption}
\theoremstyle{definition}

\newtheorem{proposition}{Proposition}
\usepackage{lineno}

\begin{document}

\begin{frontmatter}



\title{A Thermo-Kinematic Operator Splitting Framework for Stefan Problems Based on Extreme Learning Machines}

\author[1]{Siyuan Lang}
\author[1]{Zhiyue Zhang\corref{cor1}}
\ead{zhangzhiyue@njnu.edu.cn}
\cortext[cor1]{Corresponding author.}

\affiliation[1]{organization={Ministry of Education Key Laboratory for NSLSCS, School of Mathematical Sciences, Nanjing Normal University},
            city={Nanjing},
            country={China}}



\begin{abstract}
Solving Stefan problems via neural networks is challenging because of the nonlinear coupling between the solution and the free boundary, which generally leads to a non-convex optimization problem when they are parameterized and optimized simultaneously. To address this difficulty, we propose a thermo-kinematic operator splitting framework based on Extreme Learning Machines (ELMs) to decouple the solution reconstruction from the free boundary evolution. Both the solution and the complete free-boundary trajectory are represented over the full space--time domain. At each outer iteration, the solution is computed on the space--time subdomains determined by the current free-boundary trajectory, after which the interface geometry is updated through the corresponding free-boundary conditions using the computed heat fluxes. For the models considered here, each split subproblem is linear in
the unknown coefficients when the required geometric quantities are
held fixed.
Both steps are therefore formulated as convex linear least-squares
problems. The proposed framework thus avoids direct minimization of the
monolithic non-convex residual functional by isolating the nonlinear
solution--free-boundary coupling in a reduced relaxed fixed-point
iteration for the complete free-boundary trajectory. For a fixed discretization, the analysis establishes local linear
convergence of the free-boundary trajectory iteration under the stated
assumptions and with suitable relaxation. Numerical experiments in one, two, and three spatial dimensions verify the accuracy of the proposed framework on benchmark problems and illustrate its applicability to physically motivated Stefan problems.
\end{abstract}


%

\end{frontmatter}



\section{Introduction}
\label{sec:introduction}

Stefan problems provide a classical mathematical description of
phase-change processes involving moving free interfaces, with
applications including alloy solidification and crystal
growth~\cite{boettinger2002phase,langer1980instabilities,
beckermann2002modelling,sethian1992crystal}.
Mathematically, they are free-boundary problems in which the phase-wise
temperature fields, as solutions of the governing PDEs, and the free
boundary separating the phases are coupled unknowns.
The free boundary determines the evolving phase domains, while the
interfacial heat fluxes obtained from the solutions drive its motion
through the Stefan condition.
This solution--free-boundary coupling is a central difficulty in the
analysis and numerical solution of Stefan problems.

Analytical and semi-analytical approaches have provided important
solvability results and reference solutions for Stefan
problems~\cite{briozzo2007explicit,liu2012exact,
voller2004analytical,rubinshteuin1971Stefan}.
{\color{reviewertwo}
Recent studies have extended such analyses to physically motivated
models involving nonlinear thermal coefficients, internal heat sources,
Joule heating, and electrothermal effects
\cite{nauryz2026mathematical,bollati2025mathematical,
nauryz2025exact,bollati2024mathematical}.
Similarity reductions and fixed-point arguments have been used to
establish solvability for several of these models, while analytical,
integral-equation, and collocation representations have been developed
for inverse, spherical, and multiple-free-boundary Stefan problems
\cite{kassabek2024analytical,kavokin2016exact,case2012integral,
narbek2025collocation}.
These results provide both insight into particular Stefan models and
accurate reference solutions for numerical validation.
Meanwhile, numerical methods provide a complementary approach for
configurations in which analytical or semi-analytical solution
representations are not readily available.
\par}

A wide range of numerical methods has been developed for Stefan and
related free-boundary problems.
Front-tracking methods explicitly represent and evolve the moving
interface
\cite{womble1989front,unverdi1992front,juric1996front,
marshall1986front}.
Moving and parametric finite-element approaches provide related
sharp-interface discretizations
\cite{alexander1979moving,barrett2010stable,nurnberg2023structure},
while front-fixing transformations~\cite{landau1950heat},
space--time finite element methods~\cite{bonnerot1977numerical}, and
boundary-integral formulations~\cite{coleman1986boundary} provide
alternative treatments of the evolving geometry.
Enthalpy formulations represent latent-heat effects within the field
equations
\cite{voller1981accurate,voller1987implicit,date1992novel,
elliott1987error}, while finite element and variational inequality
treatments have also been developed
\cite{ciavaldini1975analyse,elliott1981finite}.
Diffuse-interface phase-field formulations
\cite{fix1982phase,mackenzie2002moving} replace the sharp phase
boundary by a transition layer, whereas level-set and related
embedded-boundary methods represent the moving interface implicitly
\cite{osher1988fronts,chen1997simple,gibou2005fourth,
papac2013level,chen2009numerical,limare2023hybrid}.
Extended finite element methods provide another sharp-interface
treatment on unfitted meshes~\cite{salvatori2009Stefan}.
These methods differ in their interface representations and bulk
discretizations, but for Stefan problems they ultimately need to
account for the coupling between the solution and the moving
free boundary.

As an alternative to conventional mesh-based discretizations,
parametric methods based on neural networks and random features have
also been developed for PDE approximation.
Physics-informed neural networks (PINNs) determine the parameters of
trial representations by minimizing residuals of the governing
equations and constraints~\cite{cybenko1989approximation,
raissi2019physics}.
Extreme learning machines (ELMs)~\cite{huang2006extreme} instead fix
the hidden-layer parameters and treat only the output coefficients as
unknowns. The resulting representation is linear in these coefficients,
so linear PDEs with linear conditions on prescribed domains lead to
linear least-squares systems
\cite{dwivedi2020physics,dong2021local}.
For Stefan problems, PINN formulations
\cite{li2023improved,madir2025physics,wang2021deep,
shkolnikov2024deep,chang2025physics}
typically use neural-network representations for the phase-wise
solutions and the free boundary and determine their parameters jointly
by minimizing a coupled residual loss.
{\color{reviewerone}
Ren et al.~\cite{ren2025physics} employed separate PIELM
representations for the solution and the free boundary, while the
corresponding coefficients were determined through a coupled nonlinear
residual formulation. In contrast, Lin et al.~\cite{lin2025discontinuous} combined ELM
approximation with front tracking in a time-marching framework.
The physical time interval is first discretized into successive time
levels. At each time step, the free boundary is advanced by the
front-tracking scheme, after which the corresponding time-discrete
field problem on the updated fixed geometry is solved by linear
least squares.
The approaches reviewed above therefore treat the
solution--free-boundary coupling either through a joint nonlinear solve
or through sequential advancement of the free boundary over discrete
physical time levels.
\par}

{\color{reviewerone}
The linearity of an ELM representation with respect to its output
coefficients does not by itself remove the nonlinearity of a Stefan
problem.
When the solutions and the free boundary are parameterized
simultaneously, the phase domains, interface evaluation locations, and
geometry-dependent interface quantities all depend on the unknown free
boundary.
Consequently, evaluating the solutions and their gradients on the
unknown free boundary couples the two sets of coefficients
nonlinearly.
The resulting monolithic least-squares problem is generally
non-convex, even when each representation is linear in its own
output coefficients.
The essential difficulty therefore lies in the
solution--free-boundary coupling rather than in the ELM representation
itself.
These observations motivate a formulation that exploits the
conditional linearity obtained for a prescribed free boundary while
treating the complete free-boundary trajectory over the prescribed
time interval as the unknown.
\par}

{\color{reviewerone}
In this work, we propose a thermo-kinematic operator splitting framework in which both the solutions and the complete free-boundary trajectory are represented over the full space--time domain.
At each outer iteration, the current free-boundary trajectory is held fixed while the phase-wise solutions are reconstructed on the
corresponding space--time subdomains.
The resulting interfacial heat fluxes are then used to update the complete trajectory.
For the models considered here, the ELM representation provides fixed
space--time features for the phase-wise solutions, so that the unknown
output coefficients enter linearly, while the free-boundary trajectory
is represented affinely in its coefficients.
Once the required geometry-dependent quantities are fixed or lagged,
the residuals in both split steps are affine in their respective
unknown coefficients.
Consequently, the thermodynamic reconstruction and the kinematic
update are both formulated as convex linear least-squares problems.
{\color{reviewertwo}
The splitting does not eliminate the nonlinear
solution--free-boundary coupling, but isolates it in a reduced
fixed-point iteration for the complete free-boundary trajectory.
The method therefore avoids direct minimization of a monolithic
non-convex residual functional; the remaining nonlinear difficulty
concerns the stability and convergence of the reduced interface map.
For a fixed discretization, we establish local linear convergence
of this iteration under the stated assumptions and with suitable
relaxation.
}
Numerical experiments in one, two, and three spatial dimensions assess
the accuracy of the reconstructed solutions and complete free-boundary
trajectories using analytical or numerical reference data.
The quantitative tests include one- and two-phase free-boundary
benchmarks, a problem with interfacial thermal resistance, Frank
similarity problems in two and three dimensions, spherical
melting--vaporization with two moving free boundaries, and spherical
solidification with a Gibbs--Thomson condition.
A quasi-stationary Mullins--Sekerka example further provides a
qualitative illustration of curvature-dependent interface evolution.
\par}

The remainder of this paper is organized as follows.
Section~\ref{sec:problem} presents the Stefan problem and its
thermo-kinematic structure.
Section~\ref{sec:method} develops the proposed full space--time
splitting framework and its ELM discretization.
Section~\ref{sec:analysis} analyzes the local convergence of the
reduced free-boundary iteration and the discrete residual consistency
at a split fixed point.
Section~\ref{sec:num} presents the numerical experiments, and
Section~\ref{sec:conclusion} concludes the paper.

\section{Problem Formulation}
\label{sec:problem}

Consider a bounded domain $\Omega\subset\mathbb{R}^{d}$ with boundary
$\partial\Omega$.
At each time $t\in(0,T]$, the domain is separated into two
time-dependent phase regions $\Omega_1(t)$ and $\Omega_2(t)$ by a free
interface $\Gamma(t)$, such that
\[
\Omega
=
\Omega_1(t)\cup\Omega_2(t)\cup\Gamma(t).
\]
The interface $\Gamma(t)$ is the unknown free boundary whose evolution
is determined as part of the solution.

{\color{reviewerone}
Let $u_i$ denote the phase-wise solutions in
$\Omega_i(t)$, $i=1,2$.
In this basic formulation, $\Omega_1(t)$ and $\Omega_2(t)$ denote the
solid and liquid phase regions, respectively.
The dimensional Stefan problem is written as
\begin{equation}
\label{eq:dimensional_stefan}
\left\{
\begin{aligned}
\rho_i c_i
\frac{\partial u_i}{\partial t}
-
\nabla\cdot(\kappa_i\nabla u_i)
&=
f_i(\mathbf{x},t),
&&
\mathbf{x}\in\Omega_i(t),
\quad i=1,2,
\\
u_i(\mathbf{x},0)
&=
u_{i,0}(\mathbf{x}),
&&
\mathbf{x}\in\Omega_i(0),
\quad i=1,2,
\\
\mathcal{B}_i u_i
&=
g_i,
&&
\mathbf{x}\in
\partial\Omega_i(t)\cap\partial\Omega,
\quad i=1,2,
\\
u_1
=
u_2
&=
u_m,
&&
\mathbf{x}\in\Gamma(t),
\\
\rho_\Gamma L V_n
&=
\llbracket \kappa\nabla u\rrbracket\cdot\mathbf n
=
\left(
\kappa_1\nabla u_1
-
\kappa_2\nabla u_2
\right)\cdot\mathbf n,
&&
\mathbf{x}\in\Gamma(t),
\\
\Gamma(0)
&=
\Gamma_0 .
\end{aligned}
\right.
\end{equation}

Here, $\rho_i$, $c_i$, and $\kappa_i$ denote the density, specific heat
capacity, and thermal conductivity of phase $i$, respectively, so that
$\rho_i c_i$ is the volumetric heat capacity.
The function $f_i$ represents a volumetric heat source, while
$\mathcal{B}_i$ denotes the boundary operator imposed on the fixed
outer boundary.

At the moving interface, $u_m$ denotes the phase-change temperature,
and the condition $u_1=u_2=u_m$ represents thermal equilibrium between
the two phases in the classical Stefan model.
The coefficient $L$ is the latent heat per unit mass, and
$\rho_\Gamma>0$ denotes the density entering the interfacial
latent-heat balance, so that $\rho_\Gamma L$ is the latent heat per
unit volume.
The vector $\mathbf n$ is the unit normal to the free boundary pointing
from $\Omega_1(t)$ to $\Omega_2(t)$, and $V_n$ denotes the corresponding
normal velocity.
The Stefan condition expresses the interfacial energy balance by
relating this velocity to the jump in the normal heat flux across the
moving interface.

\par}

For clarity of exposition, we adopt the following nondimensional
classical Stefan system as the basic model for the thermo-kinematic
splitting framework:
\begin{equation}
\label{eq:nondimensional_stefan}
\left\{
\begin{aligned}
\frac{\partial u_i}{\partial t}
-
\nabla\cdot(k_i\nabla u_i)
&=
f_i(\mathbf{x},t),
&&
\mathbf{x}\in\Omega_i(t),
\quad i=1,2,
\\
u_i(\mathbf{x},0)
&=
u_{i,0}(\mathbf{x}),
&&
\mathbf{x}\in\Omega_i(0),
\quad i=1,2,
\\
\mathcal{B}_i u_i
&=
g_i,
&&
\mathbf{x}\in
\partial\Omega_i(t)\cap\partial\Omega,
\quad i=1,2,
\\
u_1=u_2
&=
0,
&&
\mathbf{x}\in\Gamma(t),
\\
\beta V_n
&=
\llbracket \alpha\nabla u\rrbracket\cdot\mathbf n
=
\left(
\alpha_1\nabla u_1-\alpha_2\nabla u_2
\right)\cdot\mathbf n,
&&
\mathbf{x}\in\Gamma(t),
\\
\Gamma(0)
&=
\Gamma_0 .
\end{aligned}
\right.
\end{equation}

For simplicity, the same notation is used below for the corresponding
dimensionless variables. The coefficient $\beta>0$ denotes the
dimensionless latent heat parameter. The coefficients $k_i>0$ and
$\alpha_i>0$ denote the dimensionless coefficients in the bulk
diffusion operators and the Stefan condition, respectively, and need
not coincide.

{\color{reviewerone}
The homogeneous interface value in~\eqref{eq:nondimensional_stefan}
corresponds to taking the phase-change temperature as the reference
temperature in the nondimensionalization.
\par}

{\color{reviewertwo}
The basic model~\eqref{eq:nondimensional_stefan} has linear bulk
equations. If the complete free-boundary trajectory is prescribed, the
induced phase domains and interface locations are known, and the
thermodynamic problem is linear in the phase-wise solutions.
This linearity is conditional on the prescribed geometry: in the full
Stefan problem, the free-boundary trajectory is itself unknown and must
be consistent with the interfacial heat fluxes generated by the
corresponding solutions. This conditional linearity provides
the continuous structural basis for the split discretization developed
in the following sections.
\par}
{\color{reviewertwo}
The coupled structure of the Stefan problem can be expressed in thermo-kinematic form. For a prescribed interface trajectory $\Gamma(t)$, the thermodynamic
component determines the solutions in the corresponding
time-dependent domains:
\begin{equation}
\label{eq:thermodynamic_component}
\left\{
\begin{aligned}
\frac{\partial u_i}{\partial t}
-
\nabla\cdot(k_i\nabla u_i)
&=
f_i,
&&
\mathbf{x}\in\Omega_i(t),
\\
u_i(\mathbf{x},0)
&=
u_{i,0}(\mathbf{x}),
&&
\mathbf{x}\in\Omega_i(0),
\\
\mathcal{B}_i u_i
&=
g_i,
&&
\mathbf{x}\in
\partial\Omega_i(t)\cap\partial\Omega,
\\
u_1=u_2
&=
0,
&&
\mathbf{x}\in\Gamma(t).
\end{aligned}
\right.
\end{equation}

Conversely, for the resulting solutions, the same
free-boundary trajectory must satisfy the kinematic Stefan relation:
\begin{equation}
\label{eq:kinematic_component}
\left\{
\begin{aligned}
\beta V_n
&=
\llbracket \alpha\nabla u\rrbracket\cdot\mathbf n,
&&
\mathbf{x}\in\Gamma(t),
\\
\Gamma(0)
&=
\Gamma_0 .
\end{aligned}
\right.
\end{equation}

The two relations above are not independent subsystems; rather, each is
posed conditionally on information supplied by the other.
The free boundary determines the phase domains, interface locations,
and geometric quantities required for the thermodynamic problem,
whereas the resulting solutions determine the interfacial
heat fluxes entering the kinematic Stefan relation. Thus, at the
continuous level, the solution and the free boundary remain
bidirectionally coupled. The operator splitting introduced in the next
section does not remove this intrinsic coupling; it reorganizes the
numerical treatment of the coupled problem by fixing or lagging the
geometry-dependent quantities within each subproblem and retaining the
remaining nonlinearity in an outer iteration for the complete
free-boundary trajectory.
\par}

\section{Methodology}
\label{sec:method}
This section develops a numerical realization of the
thermo-kinematic structure~\eqref{eq:thermodynamic_component} and
\eqref{eq:kinematic_component} identified in
Section~\ref{sec:problem}.
We first introduce full space--time ELM representations of the
phase-wise solutions and the complete free-boundary trajectory in
Section~\ref{sec:elm_approx}.
Section~\ref{sec:splitting_algorithm} then formulates the
thermodynamic reconstruction, the kinematic update, and the resulting
reduced relaxed fixed-point iteration.
The construction below concerns the basic Stefan
model~\eqref{eq:nondimensional_stefan}; modifications for other
interface conditions are described in the corresponding numerical
examples in Section~\ref{sec:num}.

\subsection{ELM approximations}
\label{sec:elm_approx}

We parameterize both the phase-wise solutions and the complete
free-boundary trajectory over the full space--time domain using ELM
representations~\cite{huang2006extreme}, which can be interpreted as a
class of Random Feature Methods (RFMs)~\cite{rahimi2007random}.
In both representations, the internal feature parameters are randomly
generated and frozen, while only the output coefficients are treated as
unknowns.

Let $\mathbf z=(\mathbf x,t)\in\mathbb R^{d+1}$.
For each phase $i=1,2$, the solution is approximated by
\begin{equation}
\label{eq:elm_ansatz}
u_i(\mathbf{x},t)
\approx
u_i^N(\mathbf{x},t;\boldsymbol{\beta}_i)
=
\sum_{j=1}^{N}
\beta_{i,j}
\sigma(\mathbf w_{i,j}\cdot\mathbf z+b_{i,j})
=
\boldsymbol{\Phi}_i(\mathbf{x},t)
\boldsymbol{\beta}_i,
\end{equation}
where
$\boldsymbol{\beta}_i=[\beta_{i,1},\ldots,\beta_{i,N}]^\top$
is the vector of unknown output coefficients.
The internal parameters $\mathbf w_{i,j}$ and $b_{i,j}$ are randomly
generated and remain fixed throughout the computation.
The feature map is $\boldsymbol{\Phi}_i(\mathbf{x},t)
=
[\phi_{i,1}(\mathbf{x},t),\ldots,\phi_{i,N}(\mathbf{x},t)],$
with $\phi_{i,j}(\mathbf z)
=
\sigma(\mathbf w_{i,j}\cdot\mathbf z+b_{i,j})$.

In this work, we use the sine activation
$\sigma(z)=\sin(z)$, which yields a randomized Fourier-type trial space
for the smooth solutions considered here~\cite{xia2025data}.
Problem-specific fixed features may also be included without changing
the linear dependence on the output coefficients.

The complete free-boundary trajectory is represented in the same
ELM form, with the prescribed initial interface incorporated into the
trial representation:
\begin{equation}
\label{eq:interface_solution_parameterization}
\Gamma^M(\xi,t;\mathbf c)
=
\Gamma_0(\xi)
+
\sum_{m=1}^{M}
c_m\Psi_m^0(\xi,t),
\qquad
\xi\in\Xi,\quad t\in[0,T],
\end{equation}
where $\mathbf c\in\mathbb R^M$ is the vector of unknown
free-boundary coefficients and the fixed trial functions satisfy
\[
\Psi_m^0(\xi,0)=0,
\qquad
m=1,\ldots,M.
\]
Hence,
\[
\Gamma^M(\xi,0;\mathbf c)=\Gamma_0(\xi)
\]
for every $\mathbf c$, so the prescribed initial free boundary is
incorporated exactly into the trial representation.

The parameter $\xi\in\Xi\subset\mathbb R^{d-1}$ denotes the intrinsic
interface coordinate. It is omitted in one spatial dimension,
is one-dimensional in two spatial dimensions, and is two-dimensional
in three spatial dimensions.
For graph and radial representations, $\Gamma^M$ denotes the graph
function and radius, respectively, while vector-valued
parameterizations are treated componentwise.

The representation~\eqref{eq:interface_solution_parameterization} is
affine in $\mathbf c$, so the unknown output coefficients enter
linearly.
Thus, the numerical unknowns are the solution coefficient vectors
$\boldsymbol{\beta}_i$, $i=1,2$, and the free-boundary coefficient
vector $\mathbf c$.

\subsection{Thermo-kinematic operator splitting and alternating iteration}
\label{sec:splitting_algorithm}

{\color{reviewerone}
To exploit the thermo-kinematic structure described in
Section~\ref{sec:problem}, we alternate between thermodynamic solution
reconstruction on a prescribed free-boundary trajectory and a
geometry-lagged kinematic update driven by the reconstructed
interfacial heat fluxes. With the ELM features fixed and the current free-boundary trajectory
used as the fixed geometric reference in each split step, the
corresponding residuals are affine in their respective unknown output
coefficients.
Hence, both updates are formulated as convex linear least-squares
problems. The
nonlinear solution--free-boundary coupling remains in the relaxed
outer fixed-point iteration for the complete trajectory.
The complete construction and the resulting thermo-kinematic
iteration are summarized in
\cref{fig:thermo_kinematic_framework}.}

\begin{figure}[htbp]
    \centering
    \includegraphics[width=\linewidth]{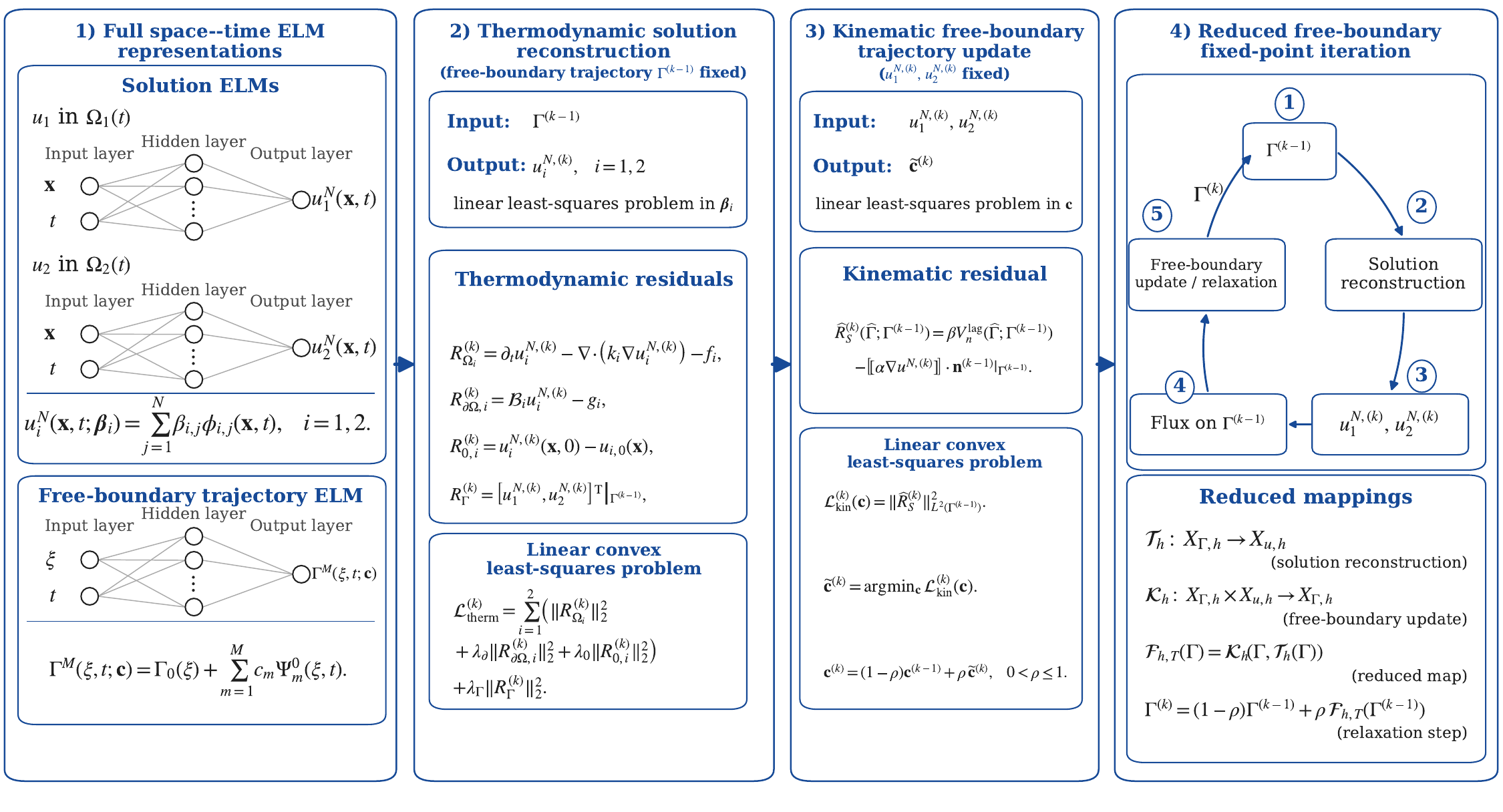}
    \caption{{\color{reviewerone}Schematic of the proposed thermo-kinematic operator splitting framework.}}
    \label{fig:thermo_kinematic_framework}
\end{figure}

The alternating iteration is initialized with a free-boundary trajectory
$\Gamma^{(0)}$ determined by an initial coefficient vector
$\mathbf c^{(0)}$.
Unless otherwise stated, the initial trajectory is taken as the
time-independent extension of the prescribed initial interface,
namely,
\[
\Gamma^{(0)}(\xi,t)=\Gamma_0(\xi),
\qquad
\xi\in\Xi,\quad t\in[0,T].
\]
This trajectory is obtained by taking $\mathbf c^{(0)}=\mathbf 0$
in~\eqref{eq:interface_solution_parameterization}.
In general, $\Gamma^{(k)}$ denotes the free-boundary trajectory
represented by the coefficient vector $\mathbf c^{(k)}$.
For each iteration step $k\geq1$, we first solve the thermodynamic
subproblem.
With the current free-boundary trajectory
\[
\Gamma^M(\xi,t;\mathbf c^{(k-1)})
\]
fixed, the free boundary is treated as a prescribed internal boundary.
Consequently, the original free-boundary problem reduces to a system of
linear parabolic initial-boundary value problems posed on the
space--time subdomains $\Omega_1^{(k-1)}$ and
$\Omega_2^{(k-1)}$ determined by the fixed trajectory
$\Gamma^{(k-1)}$.
We compute the solutions $u_1^{(k)}$ and $u_2^{(k)}$ by minimizing the
residuals of the following equations:
\begin{align}
    \frac{\partial u_i^{(k)}}{\partial t}
    -
    \nabla \cdot
    (k_i\nabla u_i^{(k)})
    &=
    f_i
    \quad
    \text{in }
    \Omega_i^{(k-1)},
    \label{eq:field_subproblem1}
    \\
    u_i^{(k)}(\mathbf{x},0)
    &=
    u_{i,0}(\mathbf{x}),
    \qquad
    \mathcal B_i u_i^{(k)}
    =
    g_i
    \quad
     \text{on }\partial\Omega_i^{(k-1)}(t)\cap\partial\Omega ,
    \\
    u_1^{(k)}
    =
    u_2^{(k)}
    &=
    0
    \quad
    \text{on }\Gamma^{(k-1)} .
    \label{eq:field_subproblem2}
\end{align}

To formulate the thermodynamic subproblem in least-squares form, we
define the residuals associated with the governing equations and
conditions.

The PDE residuals in the two phases are defined as
\begin{equation}
\label{eq:R_omega}
R_{\Omega_i}(u_i)
=
\frac{\partial u_i}{\partial t}
-
\nabla\cdot(k_i\nabla u_i)
-
f_i,
\qquad i=1,2 .
\end{equation}

The boundary and initial residuals are defined by
\begin{equation}
\label{eq:R_boundary_initial}
R_{\partial\Omega,i}(u_i)
=
\mathcal B_i u_i-g_i ,
\qquad
R_{0,i}(u_i)
=
u_i(\mathbf{x},0)-u_{i,0}(\mathbf{x}) .
\end{equation}

The interface residuals are defined by
\begin{equation}
\label{eq:R_interface}
R_{\Gamma,i}(u_i;\Gamma^{(k-1)})
=
\left.
u_i
\right|_{\Gamma^{(k-1)}},
\qquad
i=1,2 .
\end{equation}

Using the ELM approximations introduced in
Section~\ref{sec:elm_approx}, the only unknowns in the thermodynamic
step are the output coefficient vectors
$\boldsymbol{\beta}_i^{(k)}$, $i=1,2$.
For the fixed trajectory $\Gamma^{(k-1)}$, the bulk equations,
initial and outer-boundary conditions, and interface-value conditions
are linear in these coefficients.
Hence, the corresponding discrete residuals are affine in
$\boldsymbol{\beta}_i^{(k)}$, and their minimization yields a linear
least-squares problem.

We discretize this subproblem using Monte Carlo collocation with a
fixed sampling realization.
The random reference samples are generated once, and the corresponding
collocation sets are constructed from the current free-boundary
trajectory $\Gamma^{(k-1)}$.
These consist of interior points $\mathcal D_{\Omega,i}$ in the current
space--time phase domains, initial points $\mathcal D_{0,i}$ in
$\Omega_i(0)$, exterior-boundary points
$\mathcal D_{\partial,i}$ associated with phase $i$, and interface
points $\mathcal D_\Gamma$ on $\Gamma^{(k-1)}$.

Let
$\boldsymbol{\beta}
=[\boldsymbol{\beta}_1^\top,\boldsymbol{\beta}_2^\top]^\top$,
and let
$N_{\Omega,i}$, $N_{0,i}$, $N_{\partial,i}$, and $N_\Gamma$
denote the numbers of points in the corresponding collocation sets.
The thermodynamic solution coefficients $\boldsymbol{\beta}^{(k)}$ are
determined by minimizing a weighted mean-squared residual loss over
these collocation sets.
The discrete thermodynamic loss is defined as
\begin{equation}
\label{eq:discrete_loss}
\begin{aligned}
\mathcal L_{\mathrm{therm}}^{(k)}
(\boldsymbol{\beta})
={}&
\sum_{i=1}^{2}
\frac{1}{N_{\Omega,i}}
\sum_{\mathbf z\in\mathcal D_{\Omega,i}}
\left|
R_{\Omega_i}(u_i^N)(\mathbf z)
\right|^2
+
\lambda_0
\sum_{i=1}^{2}
\frac{1}{N_{0,i}}
\sum_{\mathbf z\in\mathcal D_{0,i}}
\left|
R_{0,i}(u_i^N)(\mathbf z)
\right|^2
\\
&+
\lambda_{\partial}
\sum_{i=1}^{2}
\frac{1}{N_{\partial,i}}
\sum_{\mathbf z\in\mathcal D_{\partial,i}}
\left|
R_{\partial\Omega,i}(u_i^N)(\mathbf z)
\right|^2
+
\frac{\lambda_{\Gamma}}{N_{\Gamma}}
\sum_{i=1}^{2}
\sum_{\mathbf z\in\mathcal D_{\Gamma}}
\left|
R_{\Gamma,i}
(u_i^N;\Gamma^{(k-1)})
(\mathbf z)
\right|^2 ,
\end{aligned}
\end{equation}
where $\lambda_0$, $\lambda_{\partial}$, and $\lambda_{\Gamma}$ are
positive weights used to balance the different residual terms.
At the collocation level, minimizing
\eqref{eq:discrete_loss} gives the overdetermined linear
least-squares systems
\begin{equation}
\label{eq:lsm_loss}
\boldsymbol{\beta}_i^{(k)}
=
\operatorname*{argmin}_{\boldsymbol{\beta}}
\left\|
\mathbf M_i(\mathbf c^{(k-1)})
\boldsymbol{\beta}
-
\mathbf y_i(\mathbf c^{(k-1)})
\right\|_2^2 .
\end{equation}

The global system matrix $\mathbf M_i(\mathbf c)$ and target vector $\mathbf y_i(\mathbf c)$ are constructed by stacking the discretized residual constraints:
\begin{equation}
\label{eq:matrix}
\mathbf M_i(\mathbf c)
=
\begin{bmatrix}
\mathbf M_{\Omega,i}\\
\mathbf M_{0,i}\\
\mathbf M_{\partial\Omega,i}\\
\mathbf M_{\Gamma,i}
\end{bmatrix},
\qquad
\mathbf y_i(\mathbf c)
=
\begin{bmatrix}
\mathbf f_i\\
\mathbf u_{i,0}\\
\mathbf g_i\\
\mathbf 0
\end{bmatrix}.
\end{equation}

The submatrices
$\mathbf M_{\Omega,i}$,
$\mathbf M_{0,i}$,
$\mathbf M_{\partial\Omega,i}$, and
$\mathbf M_{\Gamma,i}$
represent the discretized PDE, initial, exterior-boundary, and
phase-wise interface residuals, respectively.
Each block includes the square root of its residual weight and
sampling normalization in~\eqref{eq:discrete_loss}. The data vector
may depend on $\mathbf c$ through the collocation locations.
Since the number of collocation equations exceeds the number of
unknown solution coefficients, the resulting systems are overdetermined
and are solved in the least-squares sense.

With the updated solutions $u_i^{(k)}$, we evaluate the Stefan
residual.
Let $V_n^{(k-1)}$ denote the normal velocity associated with the current
free-boundary trajectory.
The Stefan residual is defined as
\begin{equation}
\label{eq:R_stefan}
R_S^{(k)}
=
\beta V_n^{(k-1)}
-
\llbracket \alpha\nabla u^{(k)}\rrbracket
\cdot
\mathbf n^{(k-1)} .
\end{equation}

If $\|R_S^{(k)}\|_{L^2(\Gamma^{(k-1)})}$ falls below a prescribed tolerance $\epsilon$, the outer iteration is
terminated according to this residual criterion.
Otherwise, we perform the kinematic interface update.
To distinguish the unknown updated trajectory from the trajectory
providing the lagged geometric quantities, we define the
geometry-lagged Stefan residual by
\begin{equation}
\label{eq:R_S_general}
\widehat R_S^{(k)}
\bigl(
\widehat\Gamma;
\Gamma^{(k-1)},u_1^{(k)},u_2^{(k)}
\bigr)
=
\beta
V_n^{\mathrm{lag}}
\bigl(
\widehat\Gamma;\Gamma^{(k-1)}
\bigr)
-
\left.
\llbracket
\alpha\nabla u^{(k)}
\rrbracket
\cdot
\mathbf n^{(k-1)}
\right|_{\Gamma^{(k-1)}} .
\end{equation}
Here, $\widehat\Gamma$ denotes the candidate updated trajectory, while
the interface locations and normals used to evaluate the heat flux are
taken from the current trajectory $\Gamma^{(k-1)}$.
The notation
$V_n^{\mathrm{lag}}(\widehat\Gamma;\Gamma^{(k-1)})$
indicates that the candidate trajectory is updated using
$\Gamma^{(k-1)}$ as the geometric reference.

The initial interface is satisfied exactly by
\eqref{eq:interface_solution_parameterization}, so no
initial-interface penalty term is required.
The discrete kinematic loss is defined on the interface collocation
set $\mathcal D_\Gamma$ by
\begin{equation}
\label{eq:kinematic_loss}
\mathcal L_{\mathrm{kin}}^{(k)}(\mathbf c)
=
\frac{1}{N_\Gamma}
\sum_{\mathbf z\in\mathcal D_\Gamma}
\left|
\widehat R_S^{(k)}
\bigl(
\Gamma^M(\cdot;\mathbf c);
\Gamma^{(k-1)},u_1^{(k)},u_2^{(k)}
\bigr)(\mathbf z)
\right|^2 .
\end{equation}

The intermediate coefficient vector
$\widetilde{\mathbf c}^{(k)}$ is then obtained from
\begin{equation}
\label{eq:lsm_loss_c}
\widetilde{\mathbf c}^{(k)}
=
\operatorname*{argmin}_{\mathbf c}
\mathcal L_{\mathrm{kin}}^{(k)}(\mathbf c).
\end{equation}

With the reconstructed solution fixed and
$\Gamma^{(k-1)}$ used as the geometric reference,
the residual is affine in the updated free-boundary coefficients.
Hence, the kinematic update is a linear least-squares problem with
respect to $\mathbf c$.

Finally, to control the magnitude of the interface update, we apply the relaxation scheme
\begin{equation}
\label{eq:relaxation_update}
\mathbf c^{(k)}
=
(1-\rho)\mathbf c^{(k-1)}
+
\rho\widetilde{\mathbf c}^{(k)},
\end{equation}
where $\rho\in(0,1]$ is the relaxation parameter.
The updated coefficient vector defines the updated free-boundary trajectory
$\Gamma^{(k)}$ used in the subsequent iteration.
Since~\eqref{eq:interface_solution_parameterization} is affine in
$\mathbf c$, this coefficient update is equivalent to relaxing the
complete trajectory and preserves the prescribed initial interface.

\begin{algorithm}[h]
\caption{The proposed thermo-kinematic operator splitting scheme for the basic Stefan model}
\label{alg:splitting_solver}
\begin{algorithmic}[1]

\Require
Fixed ELM features and reference collocation samples;
initial interface coefficients $\mathbf c^{(0)}$;
relaxation parameter $\rho\in(0,1]$;
tolerance $\epsilon$; maximum number of outer iterations $K_{\max}\geq1$.

\State Initialize iteration counter $k\leftarrow0$.

\While{$k<K_{\max}$}

\State $k\leftarrow k+1$

\State \textit{Step 1: Thermodynamic solution reconstruction}

\State Construct the current collocation sets from the fixed reference
samples and the trajectory determined by $\mathbf c^{(k-1)}$.

\For{$i=1,2$}

\State Construct the linear least-squares system matrix
$\mathbf M_i$ and target vector $\mathbf y_i$ from the residual constraints.

\State Solve for the output coefficients:
\[
\boldsymbol{\beta}_i^{(k)}
\leftarrow
\operatorname*{argmin}_{\boldsymbol{\beta}}
\|\mathbf M_i\boldsymbol{\beta}-\mathbf y_i\|_2^2 .
\]

\EndFor

\State \textit{Step 2: Convergence check}

\State Evaluate the Stefan residual:
\[
R_S^{(k)}
=
\beta V_n^{(k-1)}
-
\llbracket \alpha\nabla u^{(k)}
\rrbracket
\cdot\mathbf n^{(k-1)} .
\]

\If{$\|R_S^{(k)}\|_{L^2(\Gamma^{(k-1)})}<\epsilon$}
\State \textbf{break}
\EndIf

\If{$k=K_{\max}$}
\State \textbf{break}
\EndIf

\State \textit{Step 3: Kinematic interface update}

\State Solve the kinematic least-squares problem:
\[
\widetilde{\mathbf c}^{(k)}
\leftarrow
\operatorname*{argmin}_{\mathbf c}
\mathcal L_{\mathrm{kin}}^{(k)}(\mathbf c).
\]

\State Update the interface coefficients:
\[
\mathbf c^{(k)}
\leftarrow
(1-\rho)\mathbf c^{(k-1)}
+
\rho\widetilde{\mathbf c}^{(k)} .
\]

\EndWhile

\State \Return
phase-wise solutions
$u_1(\cdot;\boldsymbol{\beta}_1^{(k)})$,
$u_2(\cdot;\boldsymbol{\beta}_2^{(k)})$
and the free-boundary trajectory represented by
$\Gamma^M(\xi,t;\mathbf c^{(k-1)})$.

\end{algorithmic}
\end{algorithm}

Abstractly, the alternating procedure defined in
Algorithm~\ref{alg:splitting_solver} induces a reduced free-boundary
mapping. Let $X_{u,h}$ and $X_{\Gamma,h}$ denote the ELM approximation spaces
for the phase-wise solutions and free-boundary trajectories introduced
above, respectively.
The thermodynamic step defines a solution reconstruction map
\[
\mathcal T_h:
X_{\Gamma,h}\rightarrow X_{u,h},
\]
which maps the current free-boundary trajectory to the corresponding
solution obtained from the linear least-squares problem on the induced
space--time subdomains.

The kinematic step defines a free-boundary update map
\[
\mathcal K_h:
X_{\Gamma,h}\times X_{u,h}
\rightarrow
X_{\Gamma,h},
\]
which maps the reconstructed solution and the current free-boundary
trajectory to an updated trajectory through the Stefan condition.

The unrelaxed thermo-kinematic splitting defines the reduced
free-boundary map
\begin{equation}
\label{eq:reduced_interface_map}
\mathcal F_{h,T}(\Gamma)
=
\mathcal K_h
\bigl(
\Gamma,\mathcal T_h(\Gamma)
\bigr).
\end{equation}
Thus, one unrelaxed interface update would be $\widetilde\Gamma^{(k+1)}
=
\mathcal F_{h,T}(\Gamma^{(k)})$.

The actual outer iteration applies relaxation to this reduced map:
\begin{equation}
\label{eq:relaxed_interface_map}
\begin{aligned}
\Gamma^{(k+1)}
&=
\mathcal G_{\rho,h,T}(\Gamma^{(k)}),
\\
\mathcal G_{\rho,h,T}(\Gamma)
&=
(1-\rho)\Gamma
+
\rho\mathcal F_{h,T}(\Gamma).
\end{aligned}
\end{equation}
The local convergence properties of
$\mathcal G_{\rho,h,T}$ are analyzed in
Section~\ref{sec:analysis}.

\begin{remark}
\label{rem:convexity}

The linear least-squares structure above is conditional on the current
free-boundary trajectory.
In a monolithic parameterization, the solution and free-boundary
coefficients remain nonlinearly coupled through the moving phase
domains and interface evaluations, so the resulting monolithic
residual minimization is generally non-convex.
The splitting instead retains this nonlinearity in the reduced
free-boundary map $\mathcal F_{h,T}$, whose relaxed iteration is
analyzed in Section~\ref{sec:analysis}.

\end{remark}

\begin{remark}
\label{rem:thermal_resistance}

{\color{reviewerone}
The homogeneous interface condition of the basic Stefan
model is not essential to the splitting framework.
The thermal-resistance benchmark in Section~\ref{sec:thermal} allows
a jump in the solution across the moving interface and
uses kinetic, interfacial-energy, and modified Stefan relations in
place of the classical interface conditions.}
{\color{reviewertwo}
For the thermal-resistance model considered in
Section~\ref{sec:thermal}, the interface closure changes, but the same
thermo-kinematic splitting structure is retained: the current
trajectory is held fixed during the solution reconstruction, and the
modified Stefan relation supplies the subsequent free-boundary update.
\par}
\end{remark}

\section{Theoretical Analysis}
\label{sec:analysis}

{\color{reviewertwo}
This section analyzes the thermo-kinematic iteration for a fixed
space--time discretization. The analysis is formulated in the space
of complete free-boundary trajectories. We establish local convergence
of the relaxed reduced interface map under a local regularity
assumption and a spectral condition, and then examine the consistency
of the geometric evaluations at a split fixed point.
\par}

\subsection{Local convergence of the relaxed free-boundary fixed-point iteration}
\label{sec:convergence_analysis}

{\color{reviewertwo}
We consider a fixed time interval $[0,T]$ and fixed
trial spaces. The physical collocation locations vary with the
current free-boundary trajectory and the induced phase domains, as
described in Section~\ref{sec:splitting_algorithm}. This geometric
dependence is included in the reduced interface map. The index $h$
denotes the chosen trial spaces and collocation construction.

Trajectories are compared using the same parameterization and
parameter values. Let
\[
V_{\Gamma,h}
:=\operatorname{span}\{\Psi_1^0,\ldots,\Psi_M^0\}
\]
denote the finite-dimensional interface perturbation space generated
by the trial functions in~\eqref{eq:interface_solution_parameterization}.
For vector-valued parameterizations, this definition is applied
componentwise. Then $v(\xi,0)=0$ for every $v\in V_{\Gamma,h}$.
The admissible trajectory set satisfies
\[
X_{\Gamma,h}\subset\Gamma_0+V_{\Gamma,h},
\]
where $\Gamma_0$ also denotes its time-independent extension over
$[0,T]$. We equip $V_{\Gamma,h}$ with the norm
\[
\|v\|_\Gamma
:=
\left(
\int_0^T\int_\Xi |v(\xi,t)|^2\,\mathrm d\xi\,\mathrm dt
\right)^{1/2},
\qquad v\in V_{\Gamma,h},
\]
where $|\cdot|$ denotes the absolute value or the Euclidean norm,
as appropriate. In one spatial dimension, $\|\cdot\|_\Gamma$ is
the $L^2(0,T)$ norm.

We consider the reduced interface map $\mathcal F_{h,T}$
in~\eqref{eq:reduced_interface_map} and its relaxed counterpart
$\mathcal G_{\rho,h,T}$ in~\eqref{eq:relaxed_interface_map}.
For $0<\rho\leq1$, these maps have the same fixed trajectories.

\begin{assumption}[Local regularity of the reduced interface map]
\label{ass:interface_map_regularity}
There exists an admissible fixed trajectory $\Gamma_h^*$ of
$\mathcal F_{h,T}$ and a neighborhood $\mathcal U$ of $\Gamma_h^*$,
open relative to $\Gamma_0+V_{\Gamma,h}$ and contained in
$X_{\Gamma,h}$, such that the geometry-dependent collocation construction and the
prescribed thermodynamic and kinematic least-squares solves jointly
yield a well-defined map
\[
\mathcal F_{h,T}:\mathcal U\longrightarrow\Gamma_0+V_{\Gamma,h}.
\]
Within this neighborhood, the update depends on the outer iteration
only through the current trajectory. The map $\mathcal F_{h,T}$ is
continuously Fr\'echet differentiable on $\mathcal U$.
\end{assumption}

\begin{theorem}[Local convergence of the relaxed interface iteration]
\label{thm:interface_convergence}
Suppose that Assumption~\ref{ass:interface_map_regularity} holds,
and let
\[
\mathcal J_*
=
D\mathcal F_{h,T}(\Gamma_h^*)
:
V_{\Gamma,h}\longrightarrow V_{\Gamma,h}.
\]
Assume that
\begin{equation}
\label{eq:interface_spectral_condition}
\max_{\lambda\in\sigma(\mathcal J_*)}
\operatorname{Re}\lambda<1.
\end{equation}
Then, for the fixed discretization under consideration, there exists
$\rho_*(T)>0$ such that, for every fixed relaxation parameter satisfying
\[
0<\rho\leq1,
\qquad
\rho<\rho_*(T),
\]
there exist a neighborhood $\mathcal V_\rho\subset\mathcal U$ of
$\Gamma_h^*$ and constants $C_\rho\geq1$ and $q_\rho\in(0,1)$ with
the following property: for every
$\Gamma^{(0)}\in\mathcal V_\rho$, the iterates remain in
$\mathcal V_\rho$ and satisfy
\begin{equation}
\label{eq:interface_local_linear_convergence}
\|\Gamma^{(k)}-\Gamma_h^*\|_\Gamma
\leq
C_\rho q_\rho^{\,k}
\|\Gamma^{(0)}-\Gamma_h^*\|_\Gamma,
\qquad k\geq0.
\end{equation}
The fixed trajectory is unique within this neighborhood. The
neighborhood and constants may depend on $\rho$, $h$, and $T$, but
not on $k$.
\end{theorem}

\begin{proof}
Since $\Gamma_h^*$ is a fixed point of $\mathcal F_{h,T}$, it is
also a fixed point of $\mathcal G_{\rho,h,T}$. Differentiating the
relaxed map at $\Gamma_h^*$ gives
\[
\mathcal B_\rho
=
D\mathcal G_{\rho,h,T}(\Gamma_h^*)
=
(1-\rho)I+\rho\mathcal J_*.
\]
For each eigenvalue $\lambda_j$ of $\mathcal J_*$, the corresponding
eigenvalue of $\mathcal B_\rho$ is
\[
\mu_j(\rho)=1+\rho(\lambda_j-1).
\]
Setting $p_j(\rho)=|\mu_j(\rho)|^2$, we obtain
\[
p_j(0)=1,
\qquad
p_j'(0)=2\bigl(\operatorname{Re}\lambda_j-1\bigr)<0.
\]
Hence, $|\mu_j(\rho)|<1$ on a sufficiently small positive interval.
Since $V_{\Gamma,h}$ is finite-dimensional, there are finitely many
eigenvalues, and a common bound $\rho_*(T)>0$ can be chosen such that
\[
\operatorname{spr}(\mathcal B_\rho)
=
\max_j|\mu_j(\rho)|<1
\]
for all $0<\rho<\rho_*(T)$. Here, $\operatorname{spr}$ denotes the
spectral radius.

Fix $\rho$ satisfying the stated conditions. By the
finite-dimensional relation between spectral radius and induced
operator norms~\cite{ortega2000iterative}, there exists a norm
$\|\cdot\|_\star$, equivalent to $\|\cdot\|_\Gamma$, such that
\[
\|\mathcal B_\rho\|_\star<1.
\]
Choose $q_\rho$ with
$\|\mathcal B_\rho\|_\star<q_\rho<1$.
Continuity of $D\mathcal G_{\rho,h,T}$ implies that there exists
$r>0$ such that the closed ball
\[
B_r^\star(\Gamma_h^*)
=
\left\{
\Gamma\in\Gamma_0+V_{\Gamma,h}:
\|\Gamma-\Gamma_h^*\|_\star\leq r
\right\}
\]
is contained in $\mathcal U$ and
\[
\sup_{\Gamma\in B_r^\star(\Gamma_h^*)}
\|D\mathcal G_{\rho,h,T}(\Gamma)\|_\star
\leq q_\rho.
\]
The integral mean-value formula along the line segment joining two
trajectories in this ball then gives
\[
\|\mathcal G_{\rho,h,T}(\Gamma_1)
-
\mathcal G_{\rho,h,T}(\Gamma_2)\|_\star
\leq
q_\rho\|\Gamma_1-\Gamma_2\|_\star.
\]
In particular, since
$\mathcal G_{\rho,h,T}(\Gamma_h^*)=\Gamma_h^*$,
\[
\|\mathcal G_{\rho,h,T}(\Gamma)-\Gamma_h^*\|_\star
\leq
q_\rho\|\Gamma-\Gamma_h^*\|_\star
\leq q_\rho r<r.
\]
Thus, the closed ball is invariant under the relaxed map, which is a
contraction on this complete set. The contraction mapping theorem
gives uniqueness of the fixed trajectory in the ball and
\[
\|\Gamma^{(k)}-\Gamma_h^*\|_\star
\leq
q_\rho^{\,k}\|\Gamma^{(0)}-\Gamma_h^*\|_\star.
\]
Equivalence of the two norms yields
\eqref{eq:interface_local_linear_convergence} with a constant
$C_\rho\geq1$. Taking $\mathcal V_\rho$ to be the interior of the
above ball completes the proof.
\end{proof}

\begin{remark}[Role of relaxation]
\label{rem:interface_relaxation_scope}
Condition~\eqref{eq:interface_spectral_condition} permits eigenvalues
of the unrelaxed derivative outside the unit disk, provided that
their real parts are strictly less than one. Thus, the theorem does
not require the unrelaxed map to be locally contractive. Relaxation
modifies the linearized error modes through
$\lambda\mapsto1-\rho+\rho\lambda$, and the proof establishes a
nonempty interval of positive relaxation parameters for which all
transformed eigenvalues lie inside the unit disk. The bound
$\rho_*(T)$ gives a sufficient local range of relaxation parameters.

The admissible relaxation range may depend on the length $T$ of
the full space--time interval through the spectrum of
$D\mathcal F_{h,T}(\Gamma_h^*)$.
\end{remark}
\par}
\subsection{Discrete residual consistency at a split fixed point}
\label{sec:consistency_analysis}

{\color{reviewertwo}
We now relate a fixed trajectory of the reduced interface map to the
coupled discrete formulation of the basic Stefan model.
For $u=(u_1,u_2)$, let $\mathbf R_{\mathrm{therm},h}(u;\Gamma)$
denote the discrete thermodynamic residual vector corresponding
to~\eqref{eq:discrete_loss} on the collocation sets associated with
$\Gamma$. Let $\widehat{\mathbf R}_{\mathrm{kin},h}
(\widehat\Gamma;\Gamma,u)$ denote the geometry-lagged kinematic
residual vector from~\eqref{eq:kinematic_loss}, and let
$\mathbf R_{\mathrm{kin},h}(\Gamma,u)$ denote the corresponding
unlagged discrete Stefan residual, with all geometric quantities
evaluated on $\Gamma$.

All residual vectors incorporate the square roots of the positive
residual weights and sampling normalization factors in the
corresponding losses. The lagged and unlagged kinematic residuals
use the same collocation points and scaling.
The coupled discrete residual vector is
\begin{equation}
\label{eq:coupled_stefan_residual_vector}
\mathbf R_h(u,\Gamma)
=
\begin{bmatrix}
\mathbf R_{\mathrm{therm},h}(u;\Gamma)
\\
\mathbf R_{\mathrm{kin},h}(\Gamma,u)
\end{bmatrix}.
\end{equation}

\begin{proposition}[Residual consistency at a split fixed point]
\label{prop:split_fixed_point_consistency}
Let $\Gamma_h^*\in X_{\Gamma,h}$ satisfy
$\Gamma_h^*=\mathcal F_{h,T}(\Gamma_h^*)$, and let
$u_h^*=\mathcal T_h(\Gamma_h^*)$ be the corresponding thermodynamic
reconstruction. Then
\begin{equation}
\label{eq:fixed_point_kinematic_residual_identity}
\widehat{\mathbf R}_{\mathrm{kin},h}
(\Gamma_h^*;\Gamma_h^*,u_h^*)
=
\mathbf R_{\mathrm{kin},h}(\Gamma_h^*,u_h^*).
\end{equation}
Consequently, if the split residuals satisfy
\begin{equation}
\label{eq:split_fixed_point_residual_bounds}
\begin{aligned}
\|\mathbf R_{\mathrm{therm},h}(u_h^*;\Gamma_h^*)\|_2
&\leq\varepsilon_{\mathrm{therm}},
\\
\|\widehat{\mathbf R}_{\mathrm{kin},h}
(\Gamma_h^*;\Gamma_h^*,u_h^*)\|_2
&\leq\varepsilon_{\mathrm{kin}},
\end{aligned}
\end{equation}
for nonnegative $\varepsilon_{\mathrm{therm}}$ and
$\varepsilon_{\mathrm{kin}}$, then
\begin{equation}
\label{eq:coupled_fixed_point_residual_bound}
\|\mathbf R_h(u_h^*,\Gamma_h^*)\|_2
\leq
\left(
\varepsilon_{\mathrm{therm}}^2
+
\varepsilon_{\mathrm{kin}}^2
\right)^{1/2}.
\end{equation}
Moreover, $\Gamma_h^*(\xi,0)=\Gamma_0(\xi)$ exactly.
If both split residuals vanish, the pair $(u_h^*,\Gamma_h^*)$
satisfies the coupled Stefan equations and auxiliary conditions at
the selected collocation points.
\end{proposition}

\begin{proof}
At a split fixed point,
$\mathcal K_h(\Gamma_h^*,u_h^*)=\Gamma_h^*$, so the updated
trajectory and the geometric reference coincide. The lagged and
unlagged kinematic residuals therefore agree at the same collocation
points and with the same scaling, which proves
\eqref{eq:fixed_point_kinematic_residual_identity}.

Combining the block structure~\eqref{eq:coupled_stefan_residual_vector}
with~\eqref{eq:fixed_point_kinematic_residual_identity} yields
\[
\begin{aligned}
\|\mathbf R_h(u_h^*,\Gamma_h^*)\|_2^2
={}&
\|\mathbf R_{\mathrm{therm},h}(u_h^*;\Gamma_h^*)\|_2^2
+
\|\widehat{\mathbf R}_{\mathrm{kin},h}
(\Gamma_h^*;\Gamma_h^*,u_h^*)\|_2^2.
\end{aligned}
\]
Applying~\eqref{eq:split_fixed_point_residual_bounds} proves
\eqref{eq:coupled_fixed_point_residual_bound}. The initial interface
is satisfied by construction. If both residual blocks vanish, the
positive residual weights imply that every included equation is
satisfied at its collocation points.
\end{proof}

Hence, small split residuals at a fixed point imply approximate
satisfaction of the coupled discrete Stefan formulation, without an
additional geometry-lagging discrepancy.
\par}

\section{Numerical experiments}
\label{sec:num}

In this section, we present a series of numerical experiments to assess
the accuracy and outer-iteration behavior of the proposed
thermo-kinematic operator splitting framework.
The solution and free-boundary errors are evaluated on test sets over
the corresponding space--time domains.

Let $u_{\mathrm{pred}}(\boldsymbol{x},t)$ and
$u_{\mathrm{ref}}(\boldsymbol{x},t)$ denote the numerical solution and
the exact or numerical reference solution, respectively.
The pointwise solution error is defined as
\[
e_u(\boldsymbol{x},t)
=
u_{\mathrm{pred}}(\boldsymbol{x},t)
-
u_{\mathrm{ref}}(\boldsymbol{x},t).
\]
For a set of space--time test points
$\{(\boldsymbol{x}_m,t_m)\}_{m=1}^{N_u}$, the discrete relative
$L^2$ solution error is
\begin{equation}
\label{eq:L2_error_solution}
E_u
=
\left(
\frac{
\sum_{m=1}^{N_u}
|e_u(\boldsymbol{x}_m,t_m)|^2
}{
\sum_{m=1}^{N_u}
|u_{\mathrm{ref}}(\boldsymbol{x}_m,t_m)|^2
}
\right)^{1/2}.
\end{equation}
The corresponding error evaluated on spatial test points at a fixed
time $t$ is denoted by $E_u(t)$.
For multiphase problems, the test points are assigned to the phase
regions determined by the reference free boundary, and
$u_{\mathrm{pred}}$ at each test point is taken from the corresponding
phase-wise numerical approximation.
Phase-wise errors $E_{u_i}$ are obtained by applying
\eqref{eq:L2_error_solution} separately to the test points in phase
$i$.
The quantity $|e_u|$ denotes the pointwise absolute solution error.

{\color{reviewerone}
For the free boundary, let $\Gamma_{\mathrm{pred}}(\xi,t)$ and
$\Gamma_{\mathrm{ref}}(\xi,t)$ denote the numerical and reference
representations of the complete trajectory, where $\xi\in\Xi$ is the
intrinsic interface parameter introduced in
Section~\ref{sec:elm_approx}.
The numerical and reference trajectories are expressed in the same
geometric coordinates and compared at common values of the intrinsic
parameter and time.
The pointwise free-boundary error is defined as
}
\[
e_\Gamma(\xi,t)
=
\Gamma_{\mathrm{pred}}(\xi,t)
-
\Gamma_{\mathrm{ref}}(\xi,t).
\]
{\color{reviewerone}
For test parameter pairs
$\{(\xi_m,t_m)\}_{m=1}^{N_\Gamma}$ spanning the time interval of the
corresponding example, the discrete relative $L^2$ error of the
complete free-boundary trajectory is}
\begin{equation}
\label{eq:L2_error_interface}
E_\Gamma
=
\left(
\frac{
\sum_{m=1}^{N_\Gamma}
\|e_\Gamma(\xi_m,t_m)\|_2^2
}{
\sum_{m=1}^{N_\Gamma}
\|\Gamma_{\mathrm{ref}}(\xi_m,t_m)\|_2^2
}
\right)^{1/2}.
\end{equation}
{\color{reviewerone}
Here, $\|\cdot\|_2$ is the Euclidean norm for vector-valued
representations and reduces to the absolute value for scalar
positions or graph functions.
The intrinsic parameter $\xi$ is omitted in one spatial dimension;
for a graph interface $x=\Gamma(y,t)$, it is given by $\xi=y$.
The quantity $\|e_\Gamma\|_2$ denotes the pointwise absolute
free-boundary error.
For problems with multiple free boundaries, $E_{\Gamma_j}$ denotes
the error obtained by applying \eqref{eq:L2_error_interface}
separately to interface $\Gamma_j$.
When an iteration-dependent error is reported,
$E_\Gamma^{(k)}$ denotes \eqref{eq:L2_error_interface} evaluated with
$\Gamma_{\mathrm{pred}}=\Gamma^{(k)}$.
\par}

\subsection{One-dimensional one-phase free-boundary benchmark}
\label{sec:1D1phase}

{\color{reviewerone}
We begin with a one-dimensional one-phase free-boundary
benchmark~\cite{furzeland1980comparative,wang2021deep,ren2025physics}.
This test prescribes Dirichlet and Neumann data at the moving boundary and is included for numerical verification.}
For $t\in(0,1]$, the solution on
$\Omega(t)=[0,\Gamma(t)]$ satisfies
$u_t-u_{xx}=0$ with $u_x(0,t)=2$.
At the moving boundary,
$u(\Gamma(t),t)=0$ and
$u_x(\Gamma(t),t)=\sqrt{3-2t}$.
The exact solution and free-boundary trajectory are
\[
u_{\mathrm{ref}}(x,t)
=
-\frac{x^2}{2}+2x-\frac{1}{2}-t,
\qquad
\Gamma_{\mathrm{ref}}(t)=2-\sqrt{3-2t}.
\]
The initial data are prescribed from the reference solution, with
$\Gamma(0)=2-\sqrt{3}$.

Figure~\ref{fig:case1_results} shows the pointwise free-boundary
and solution errors, both of order $10^{-13}$.
The corresponding relative $L^2$ errors $E_u$ and $E_\Gamma$
are reported in Table~\ref{tab:error_comparison}.

\begin{table}[h]
  \centering
  \caption{Relative $L^2$ errors $E_u$ and $E_\Gamma$ for the four
verification benchmarks. The PINN and PIELM values are taken from
Ren et al.~\cite{ren2025physics}; the PINN results correspond to the
formulation of Wang and Perdikaris~\cite{wang2021deep}.}
  \label{tab:error_comparison}
  \resizebox{\linewidth}{!}{%
  \begin{tabular}{lcccccccc}
    \toprule
    \multirow{2}{*}{Case}
    & \multicolumn{2}{c}{PINN~\cite{wang2021deep}}
    & \multicolumn{2}{c}{PIELM~\cite{ren2025physics}}
    & \multicolumn{2}{c}{
      {\color{reviewerone}Monolithic ELM--GN}}
    & \multicolumn{2}{c}{Proposed framework} \\
    \cmidrule(lr){2-3}
    \cmidrule(lr){4-5}
    \cmidrule(lr){6-7}
    \cmidrule(lr){8-9}
    & $E_u$ & $E_\Gamma$
    & $E_u$ & $E_\Gamma$
    & {\color{reviewerone}$E_u$}
    & {\color{reviewerone}$E_\Gamma$}
    & $E_u$ & $E_\Gamma$ \\
    \midrule
    Section~\ref{sec:1D1phase}
    & 3.92e-04 & 1.04e-03
    & 3.44e-08 & 7.39e-07
    & {\color{reviewerone}5.65e-08}
    & {\color{reviewerone}6.70e-08}
    & 2.80e-14 & 3.27e-14 \\
    Section~\ref{sec:1D2phase}, Case~1
    & 7.34e-04 & 3.52e-04
    & 1.04e-07 & 4.26e-08
    & {\color{reviewerone}1.66e-08}
    & {\color{reviewerone}5.24e-09}
    & 1.52e-14 & 3.16e-15 \\
    Section~\ref{sec:1D2phase}, Case~2
    & 3.39e-04 & 4.14e-05
    & 3.19e-06 & 8.15e-07
    & {\color{reviewerone}1.22e-08}
    & {\color{reviewerone}3.91e-09}
    & 4.79e-13 & 3.12e-14 \\
    Section~\ref{sec:2D2phase}
    & 3.70e-03 & 3.70e-03
    & 9.40e-06 & 2.78e-06
    & {\color{reviewerone}7.69e-08}
    & {\color{reviewerone}3.92e-08}
    & 4.57e-13 & 7.18e-14 \\
    \bottomrule
  \end{tabular}}
\end{table}

\begin{figure}[h]
    \centering
    \begin{subfigure}[b]{0.48\textwidth}
        \centering
        \includegraphics[width=\textwidth]
        {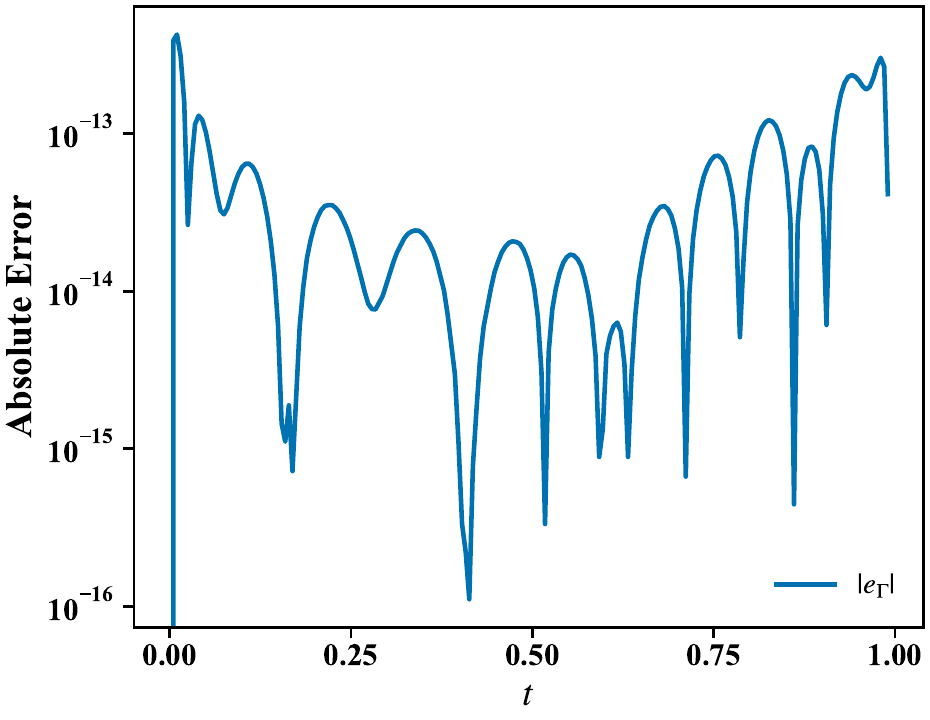}
        \caption{Free-boundary error $|e_\Gamma(t)|$.}
        \label{fig:boundary_error_exp1}
    \end{subfigure}
    \hfill
    \begin{subfigure}[b]{0.48\textwidth}
        \centering
        \includegraphics[width=\textwidth]
        {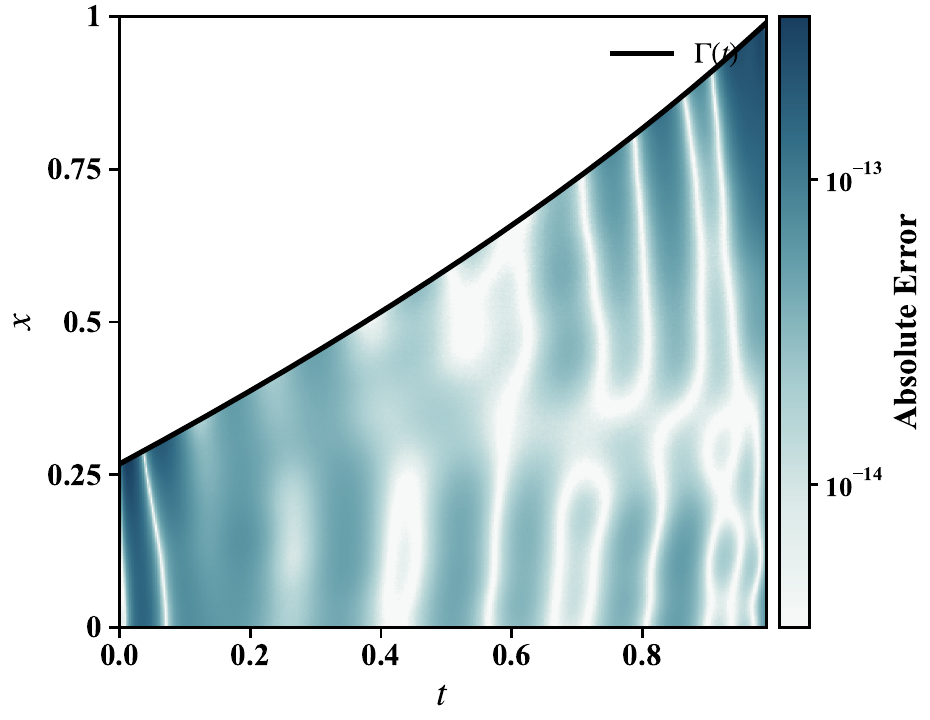}
        \caption{Solution error $|e_u(x,t)|$.}
        \label{fig:solution_error_exp1}
    \end{subfigure}
    \caption{Errors for the problem in Section~\ref{sec:1D1phase}.}
    \label{fig:case1_results}
\end{figure}

\subsection{One-dimensional two-phase Stefan problem}
\label{sec:1D2phase}

We next consider two one-dimensional two-phase Stefan
benchmarks over $t\in[0,1]$.
In both cases, the solid region $\Omega_1(t)$ lies to the right of
$\Gamma(t)$ and the liquid region $\Omega_2(t)$ to the left,
consistent with the phase convention in Section~\ref{sec:problem}.
The phase-wise solutions satisfy
$u_{i,t}-k_i u_{i,xx}=0$, $i=1,2$, with
$u_1(\Gamma(t),t)=u_2(\Gamma(t),t)=0$.
Thus, $n=-1$ and $V_n=-\Gamma'(t)$. With $\beta=1$, the
Stefan condition is
$\Gamma'(t)=\alpha_1u_{1,x}(\Gamma(t),t)
-\alpha_2u_{2,x}(\Gamma(t),t)$.

\subsubsection*{Case 1.}

{\color{reviewerone}
The first benchmark uses an exponential two-phase analytical solution
previously considered by Johansson et
al.~\cite{johansson2014meshless}, Wang and
Perdikaris~\cite{wang2021deep}, and Ren et
al.~\cite{ren2025physics}.}
The fixed spatial domain is $[0,2]$, with
$\Omega_1(t)=[\Gamma(t),2]$ and
$\Omega_2(t)=[0,\Gamma(t)]$.
We set $k_1=\alpha_1=1$ and $k_2=\alpha_2=2$, giving
$\Gamma'(t)
=u_{1,x}(\Gamma(t),t)-2u_{2,x}(\Gamma(t),t)$.
The exact solution and free-boundary trajectory are
\[
u_{1,\mathrm{ref}}(x,t)=e^{\frac{2t-2x+1}{2}}-1,
\qquad
u_{2,\mathrm{ref}}(x,t)=2e^{\frac{2t-2x+1}{4}}-2,
\qquad
\Gamma_{\mathrm{ref}}(t)=t+0.5.
\]
The initial and fixed-boundary data are prescribed from the reference
solution, with $\Gamma(0)=0.5$.

Figure~\ref{fig:case3_results} shows the pointwise free-boundary
and solution errors, of orders $10^{-14}$ and $10^{-13}$,
respectively.
The corresponding relative $L^2$ errors $E_u$ and $E_\Gamma$
are reported in Table~\ref{tab:error_comparison}.

\begin{figure}[h]
    \centering
    \begin{subfigure}[b]{0.48\textwidth}
        \centering
        \includegraphics[width=\textwidth]
        {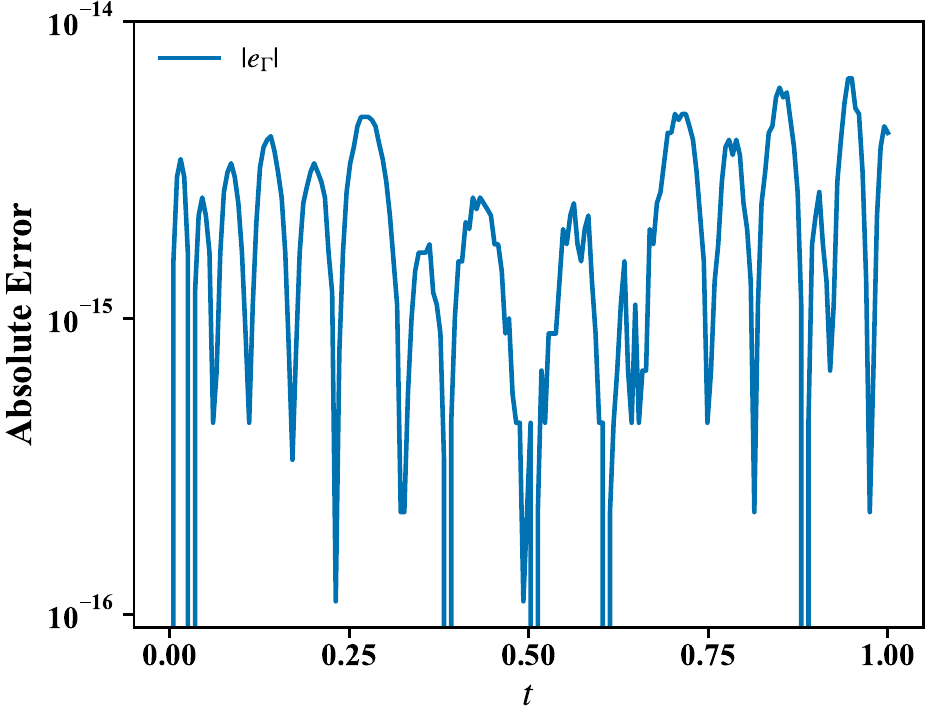}
        \caption{Free-boundary error $|e_\Gamma(t)|$.}
        \label{fig:2phase_boundary_error}
    \end{subfigure}
    \hfill
    \begin{subfigure}[b]{0.48\textwidth}
        \centering
        \includegraphics[width=\textwidth]
        {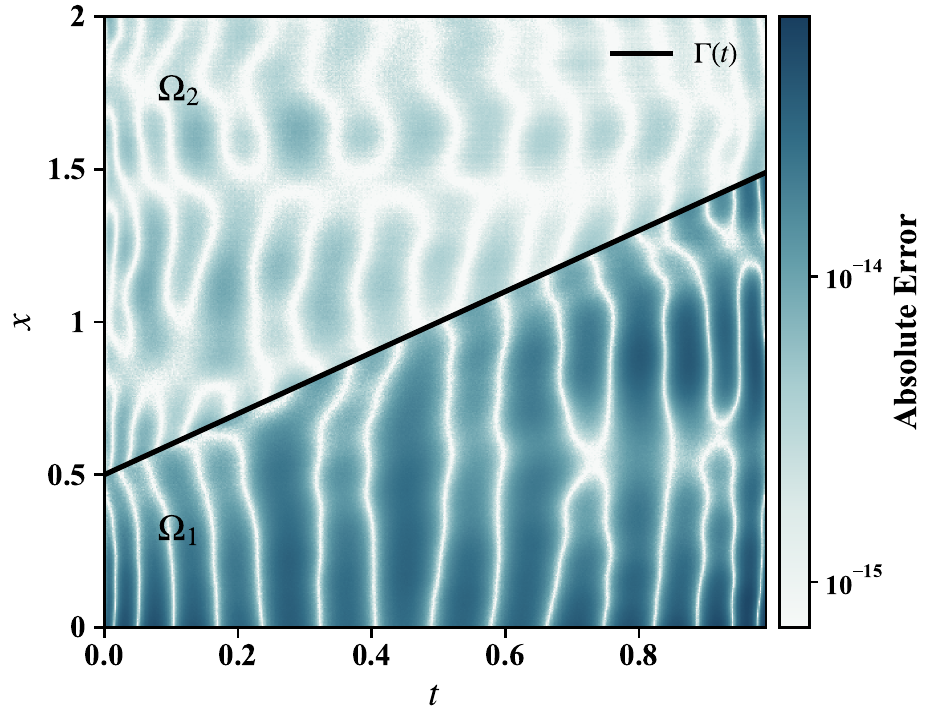}
        \caption{Solution error $|e_u(x,t)|$.}
        \label{fig:2phase_heatmap}
    \end{subfigure}
    \caption{Errors for Case~1 in Section~\ref{sec:1D2phase}.}
    \label{fig:case3_results}
\end{figure}

{\color{reviewertwo}
To examine the dependence of the outer iteration on the length of the
full space--time interval, we repeat Case~1 for several terminal times
$T$. In this auxiliary test, the fixed right boundary is set to
$x_{\max}=3$, while the numbers of random features and collocation
samples are kept unchanged. For each $T$, we record the largest tested
relaxation parameter for which convergence is observed.

\begin{table}[htbp]
\centering
\caption{Largest tested relaxation parameter yielding convergence for
different full space--time interval lengths in
Section~\ref{sec:1D2phase}, Case~1.}
\label{tab:rho_T_case1}
\begin{tabular}{c|ccccc}
\toprule
$T$
& 0.1 & 0.5 & 1.0 & 1.5 & 2.0
\\
\midrule
$\rho_{\max}$
& 1.0000 & 0.9875 & 0.5375 & 0.2500 & 0.1375
\\
\bottomrule
\end{tabular}
\end{table}

As shown in Table~\ref{tab:rho_T_case1}, the largest tested
relaxation parameter for which convergence is observed decreases
markedly as the full space--time interval is extended.
For this benchmark, the result illustrates that the observed
convergence range can depend on the length of the full space--time
interval. This is consistent with
Remark~\ref{rem:interface_relaxation_scope}, where the admissible
local relaxation range may depend on $T$ through the spectrum of
$D\mathcal F_{h,T}(\Gamma_h^*)$.
\par}
\subsubsection*{Case 2.}

{\color{reviewerone}
The second benchmark uses a time-shifted error-function similarity
solution from the classical Neumann family of two-phase Stefan
problems~\cite{furzeland1980comparative,mitchell2010application,
johansson2013meshless}.}
The fixed spatial domain is $[0,1.5]$, with
$\Omega_1(t)=[\Gamma(t),1.5]$ and
$\Omega_2(t)=[0,\Gamma(t)]$.
We set $k_1=2$, $k_2=1$, and $\alpha_1=\alpha_2=1$, giving
\[
\Gamma'(t)
=u_{1,x}(\Gamma(t),t)-u_{2,x}(\Gamma(t),t).
\]
The exact solution and free-boundary trajectory are
\[
\begin{aligned}
u_{1,\mathrm{ref}}(x,t)
&=-1+\frac{\operatorname{erfc}\!\left(x/(2\sqrt{2t+1})\right)}
{\operatorname{erfc}(\alpha/\sqrt{2})},\\
u_{2,\mathrm{ref}}(x,t)
&=1-\frac{\operatorname{erf}\!\left(x/\sqrt{4t+2}\right)}
{\operatorname{erf}(\alpha)},\\
\Gamma_{\mathrm{ref}}(t)&=\alpha\sqrt{4t+2},
\end{aligned}
\]
where the similarity parameter $\alpha$ satisfies
\[
\alpha\sqrt{\pi}
-\frac{e^{-\alpha^2}}{\operatorname{erf}(\alpha)}
+\frac{1}{\sqrt{2}}
\frac{e^{-\alpha^2/2}}{\operatorname{erfc}(\alpha/\sqrt{2})}
=0.
\]
The initial and fixed-boundary data are prescribed from the reference
solution, with $\Gamma(0)=\alpha\sqrt{2}$.

Figure~\ref{fig:case4_results} shows the pointwise free-boundary
and solution errors.
The corresponding relative $L^2$ errors $E_u$ and $E_\Gamma$
are reported in Table~\ref{tab:error_comparison}.

\begin{figure}[h]
    \centering
    \begin{subfigure}[b]{0.48\textwidth}
        \centering
        \includegraphics[width=\textwidth]
        {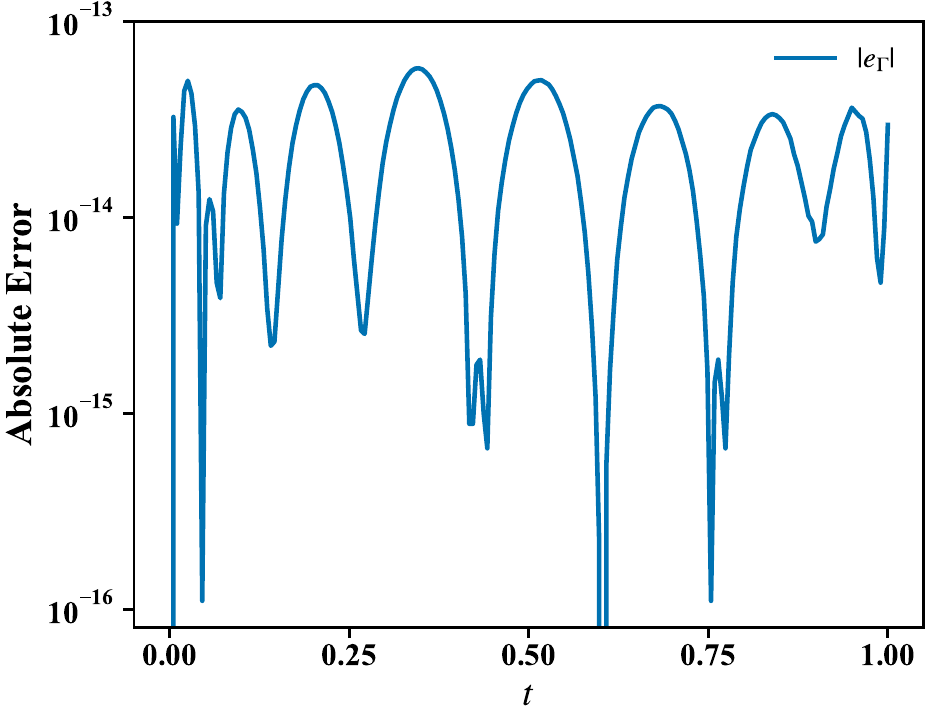}
        \caption{Free-boundary error $|e_\Gamma(t)|$.}
        \label{fig:2phase_boundary_error2}
    \end{subfigure}
    \hfill
    \begin{subfigure}[b]{0.48\textwidth}
        \centering
        \includegraphics[width=\textwidth]
        {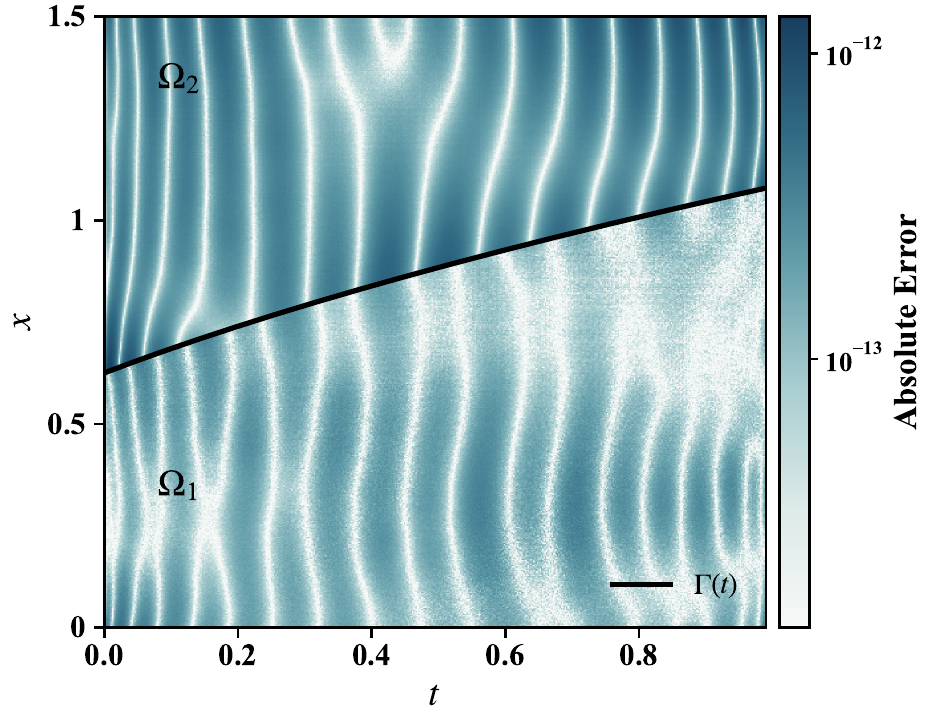}
        \caption{Solution error $|e_u(x,t)|$.}
        \label{fig:2phase_heatmap2}
    \end{subfigure}
    \caption{Errors for Case~2 in Section~\ref{sec:1D2phase}.}
    \label{fig:case4_results}
\end{figure}

For $\rho=0.5$, Figure~\ref{fig:case4_traject}(a)--(b) shows
how the complete free-boundary trajectory approaches the reference
trajectory, together with the corresponding decay of the Stefan
residual and $E_\Gamma^{(k)}$.
The latter reaches the level of $10^{-12}$ after approximately
$16$ outer iterations.

Figure~\ref{fig:case4_traject}(c) compares several relaxation
parameters for which convergence is observed in this example.
The approximately linear decay of $E_\Gamma^{(k)}$ on the
semi-logarithmic scale is consistent with the local linear convergence
behavior discussed in Section~\ref{sec:convergence_analysis}, while
the different slopes illustrate the influence of $\rho$ on the
outer-iteration rate.
A neuron-number study for this case is provided in
\ref{app:case2_neuron_study}.

\begin{figure}[h]
    \centering
    \begin{subfigure}[t]{0.32\textwidth}
        \centering
        \includegraphics[width=\textwidth]
        {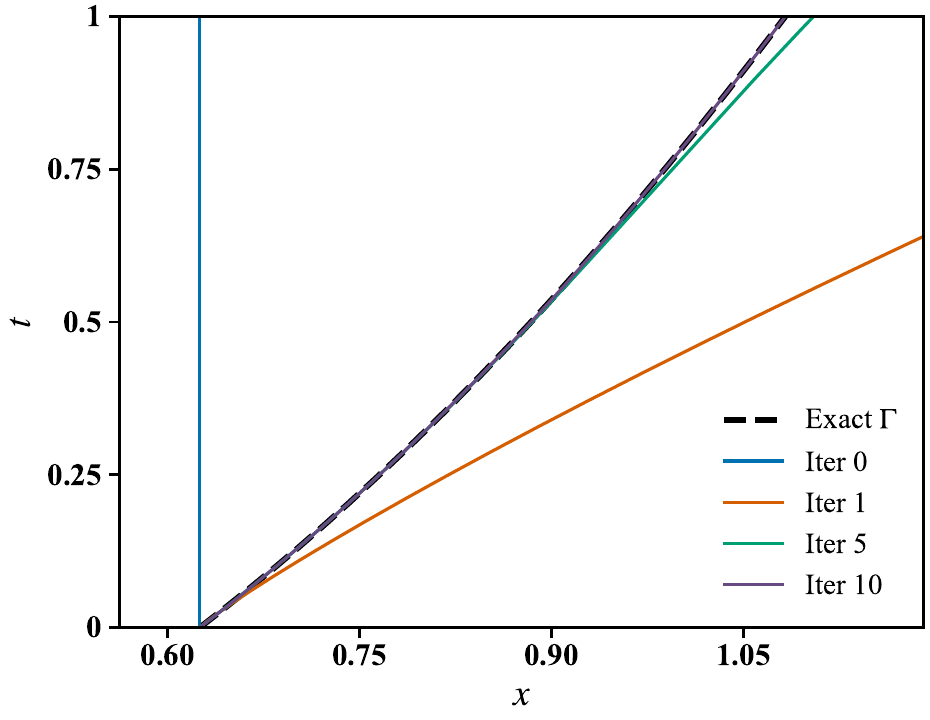}
        \caption{Free-boundary trajectory iterates.}
        \label{fig:2phase_trajectory}
    \end{subfigure}
    \hfill
    \begin{subfigure}[t]{0.32\textwidth}
        \centering
        \includegraphics[width=\textwidth]
        {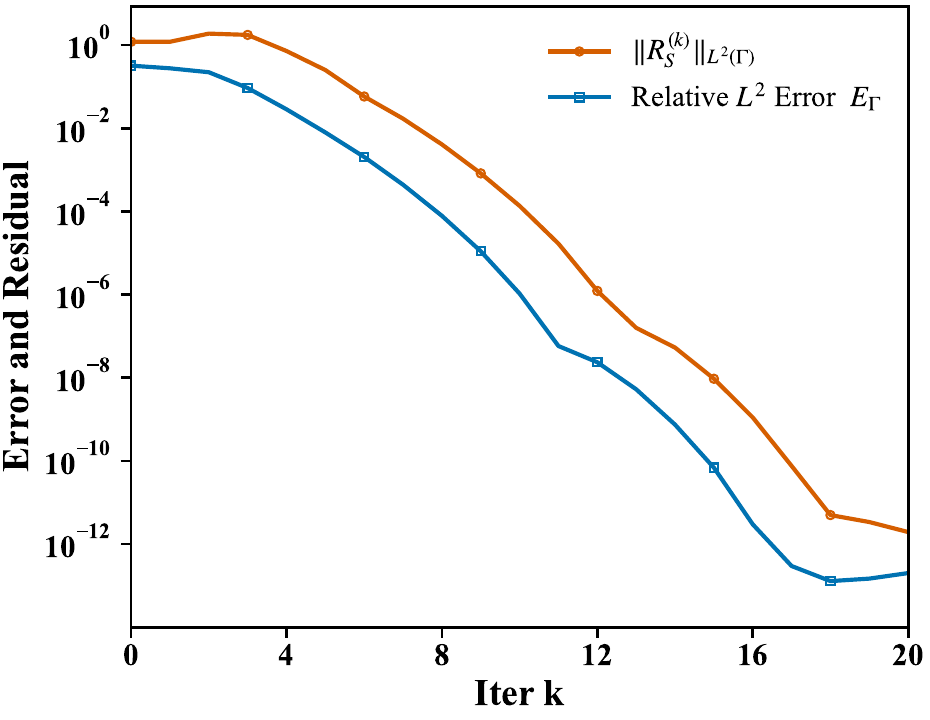}
        \caption{Stefan residual and free-boundary error.}
        \label{fig:2phase_convergence}
    \end{subfigure}
    \hfill
    \begin{subfigure}[t]{0.32\textwidth}
        \centering
        \includegraphics[width=\textwidth]
        {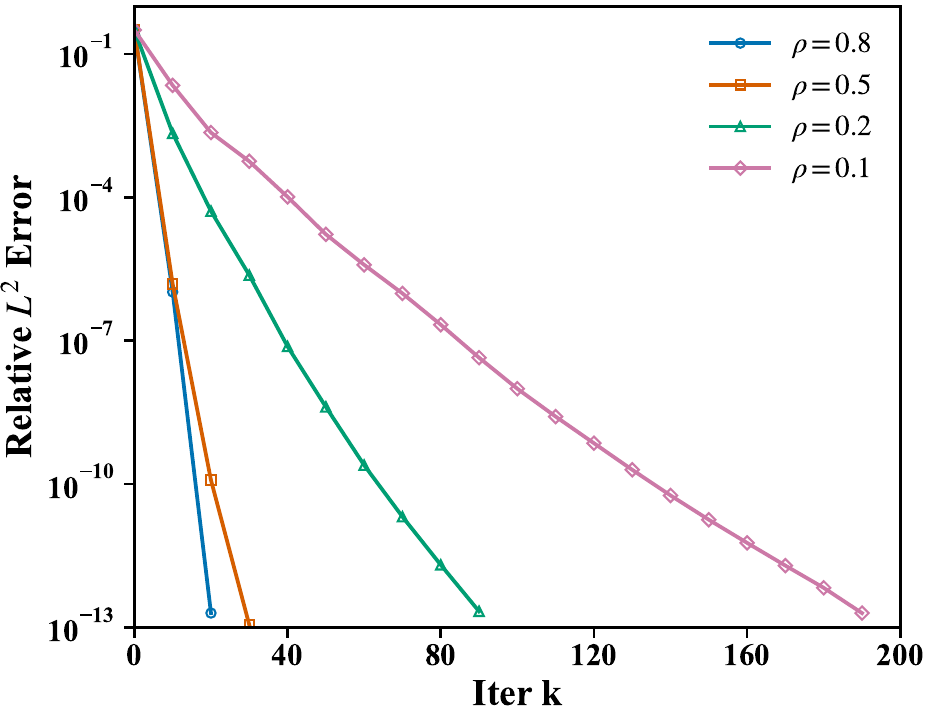}
        \caption{Free-boundary error for different $\rho$.}
        \label{fig:rho}
    \end{subfigure}
    \caption{Outer-iteration behavior for Case~2 in Section~\ref{sec:1D2phase}.}
    \label{fig:case4_traject}
\end{figure}

\subsection{Two-dimensional two-phase Stefan problem}
\label{sec:2D2phase}

{\color{reviewerone}We consider a dimensionless two-dimensional two-phase Stefan
benchmark~\cite{wang2021deep,ren2025physics} on
$\Omega=[0,2]\times[0,1]$ over $t\in[0,1]$.}
The free boundary $x=\Gamma(y,t)$ separates the solid region
$\Omega_1(t)=\{(x,y)\in\Omega:x\geq\Gamma(y,t)\}$
from the liquid region
$\Omega_2(t)=\{(x,y)\in\Omega:x\leq\Gamma(y,t)\}$.
The phase-wise solutions satisfy
$u_{i,t}-k_i\Delta u_i=0$, $i=1,2$, with
$k_1=\alpha_1=1$, $k_2=\alpha_2=2$, and $\beta=1$.
The normal and normal velocity are
\[
\boldsymbol n=\frac{(-1,\Gamma_y)}{\sqrt{1+\Gamma_y^2}},
\qquad
V_n=-\frac{\Gamma_t}{\sqrt{1+\Gamma_y^2}}.
\]
On $\Gamma(t)$, we impose $u_1=u_2=0$ and
$\beta V_n=\llbracket\alpha\nabla u\rrbracket\cdot\boldsymbol n$.
The exact solution and free-boundary trajectory are
\[
u_{1,\mathrm{ref}}(x,y,t)=e^{\frac{2t-2x+1}{2}}-1,
\qquad
u_{2,\mathrm{ref}}(x,y,t)=2e^{\frac{2t-2x+1}{4}}-2,
\qquad
\Gamma_{\mathrm{ref}}(y,t)=t+0.5.
\]
The initial and exterior-boundary data are prescribed from the reference solution, with $\Gamma(y,0)=0.5$.

Figure~\ref{fig:exp3_error_snapshots} shows pointwise solution errors
of order $10^{-12}$ in both subdomains.

\begin{figure}[h]
    \centering
    \includegraphics[width=0.95\textwidth]
    {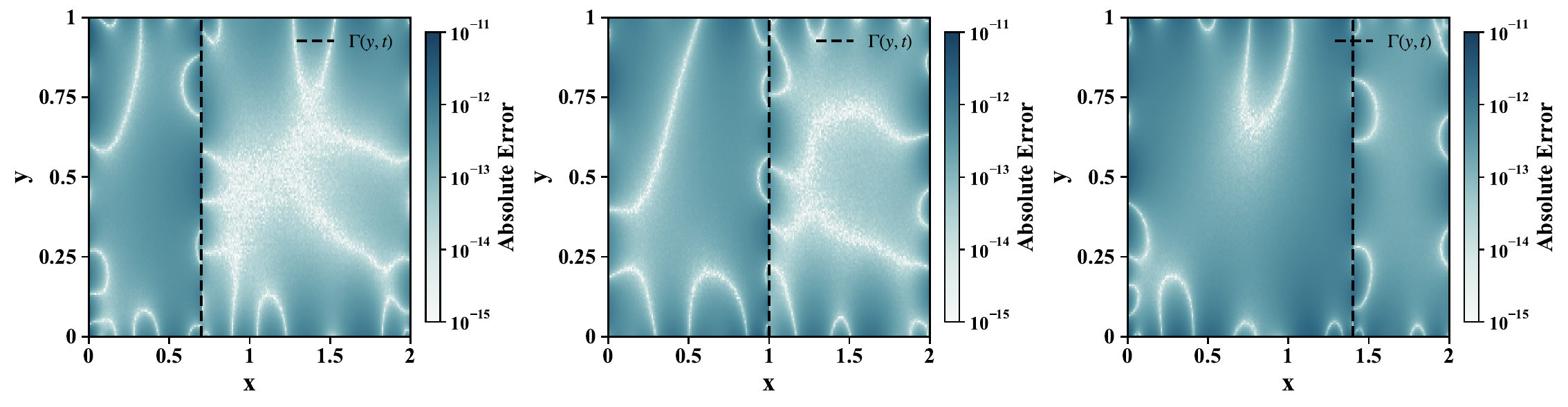}
    \caption{Solution error $|e_u(x,y,t)|$ at $t=0.20$, $0.50$, and $0.90$
for the problem in Section~\ref{sec:2D2phase}.}
    \label{fig:exp3_error_snapshots}
\end{figure}

In Figure~\ref{fig:exp3_convergence}(a), the fixed-time relative
$L^2$ solution error $E_u(t)$ remains of order $10^{-13}$.
Figure~\ref{fig:exp3_convergence}(b) shows the predicted complete
free-boundary trajectory, with pointwise absolute free-boundary errors
of the same order.
The corresponding discrete space--time relative errors $E_u$ and
$E_\Gamma$ are reported in Table~\ref{tab:error_comparison}.

\begin{figure}[h]
    \centering
    \begin{subfigure}[b]{0.45\textwidth}
        \centering
        \includegraphics[width=\textwidth]
        {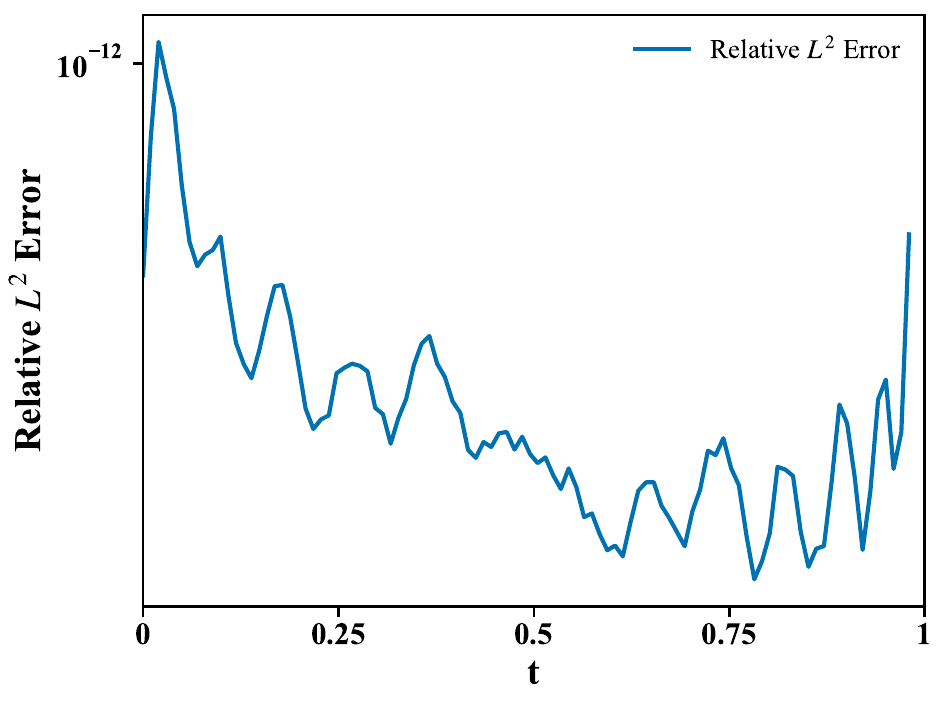}
        \caption{Fixed-time relative $L^2$ solution error $E_u(t)$.}
        \label{fig:exp3_l2_error}
    \end{subfigure}
    \hfill
    \begin{subfigure}[b]{0.5\textwidth}
        \centering
        \includegraphics[width=\textwidth]
        {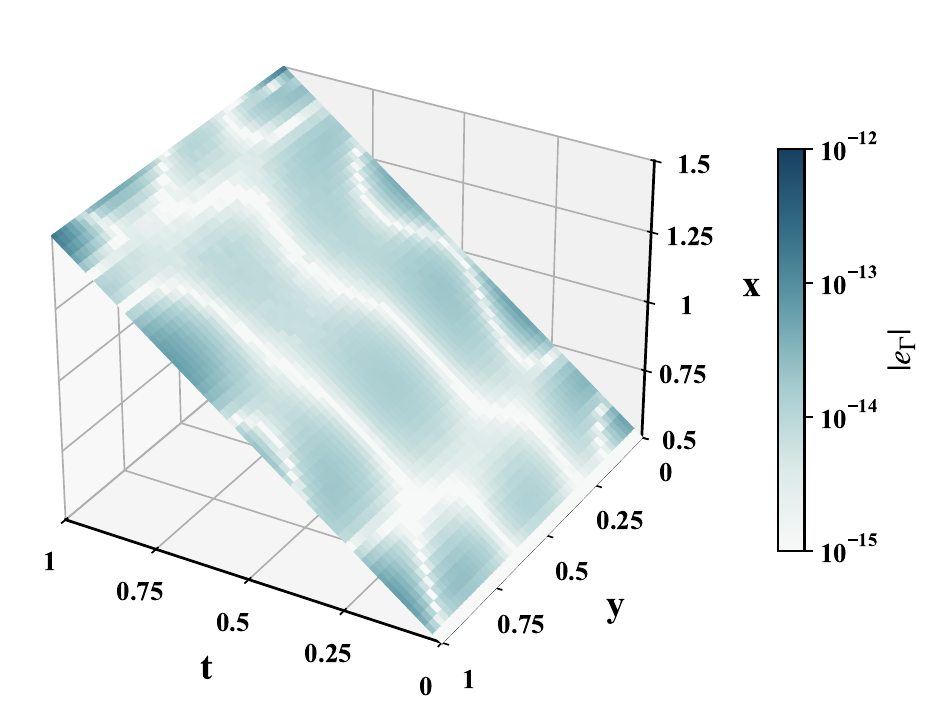}
        \caption{Free-boundary trajectory colored by $|e_\Gamma(y,t)|$.}
        \label{fig:exp3_interface_3d}
    \end{subfigure}
    \caption{Solution and free-boundary errors for the problem in
Section~\ref{sec:2D2phase}.}
    \label{fig:exp3_convergence}
\end{figure}

{\color{reviewerone}
Table~\ref{tab:error_comparison} compares the four verification
benchmarks using PINN, PIELM, the monolithic ELM--GN baseline, and
the proposed framework.
The PINN and PIELM error values are taken from
Ren et al.~\cite{ren2025physics}; the PINN results correspond to the
formulation of Wang and Perdikaris~\cite{wang2021deep}.
In these literature formulations, the solution and free-boundary
parameters are determined jointly through a coupled residual
functional.

For the matched-discretization comparison, the monolithic
ELM--GN baseline uses the same solution and free-boundary
representations, frozen random features, numbers of neurons,
collocation settings, residual weights, and initial free-boundary
trajectory as the proposed framework.
The solution and free-boundary coefficients are determined jointly by
applying Gauss--Newton iterations to the coupled nonlinear residual
system.
In contrast, the proposed framework uses two convex linear
least-squares subproblems linked by a reduced nonlinear fixed-point
iteration for the complete free-boundary trajectory.

Across the four benchmarks, the proposed framework attains relative
errors ranging from $3.16\times10^{-15}$ to
$4.79\times10^{-13}$, with both $E_u$ and $E_\Gamma$
approximately four to six orders of magnitude smaller than the
corresponding monolithic ELM--GN errors.
These results support an accuracy advantage of the thermo-kinematic
splitting under the matched-discretization settings tested here.
The PINN and PIELM values are literature results obtained under their
respective numerical settings and are included as additional
reference values for comparison.
\par}
\subsection{Stefan problem with interfacial thermal resistance}
\label{sec:thermal}

{\color{reviewerone}
We next consider a dimensionless two-phase Stefan problem with
interfacial thermal resistance based on the reduced interface relations
of Buchbinder and Galenko~\cite{BuchbinderGalenko2018}.
The problem is posed on $\Omega=[0,1]^2$ for $t\in[0,1]$, with
\[
\Omega_1(t)
=
\{(x,y)\in\Omega:x<\Gamma(y,t)\},
\qquad
\Omega_2(t)
=
\{(x,y)\in\Omega:x>\Gamma(y,t)\}.
\]
The two regions are separated by the free boundary
$x=\Gamma(y,t)$.
Unlike the basic model in Section~\ref{sec:problem}, the solution is
discontinuous across the interface and the classical continuity
condition is replaced by a coupled thermodynamic interface closure.
\par}

Following Section~\ref{sec:problem}, the phase-wise solutions satisfy
$u_{i,t}-\nabla\cdot(k_i\nabla u_i)=0$ in $\Omega_i(t)$,
$i=1,2$.
We write
$q_i=-\alpha_i\nabla u_i\cdot\boldsymbol n$, where
$\boldsymbol n$ points from $\Omega_1(t)$ to $\Omega_2(t)$.
The interface is governed by
\begin{equation}
\label{eq:Stefan_resistance}
\left\{
\begin{aligned}
&V_n
=
\frac{L_1}{2}(u_1+u_2-2u_m)
+
L_2(u_1-u_2),
\\[2pt]
&\frac{q_1+q_2}{2}
-\frac{H}{2}V_n
=
\frac{L_3}{2}(u_1+u_2-2u_m)
+
L_4(u_1-u_2),
\\[2pt]
&[\beta+c(u_2-u_1)]V_n
=
q_2-q_1,
\end{aligned}
\right.
\qquad
\text{on }\Gamma(t).
\end{equation}

Here, $\beta$ is the dimensionless latent-heat parameter,
$c$ is the coefficient of the temperature-jump contribution in the
modified Stefan balance, $u_m$ is the equilibrium melting temperature,
$H$ is the prescribed interfacial enthalpy coefficient, and
$L_1,\ldots,L_4$ are kinetic coefficients.
The corresponding Kapitza resistance is
$R_{\mathrm{th}}=L_4^{-1}$.
These three relations form the complete interface closure; no
additional independent temperature-jump condition is imposed.

We set
$k_1=k_2=\alpha_1=\alpha_2=\beta=c=1$,
$T_0=1$, $u_m=3$,
$L_1=-8/15$, $L_2=-1/10$, $L_3=0$, $L_4=5$, and
$H=26/5$, so that $R_{\mathrm{th}}=0.2$.
{\color{reviewerone}
The reference solution is constructed from the planar traveling-wave
form of Ref.~\cite{BuchbinderGalenko2018}.
With
$\eta(\boldsymbol{x},t)=(x+y)/\sqrt{2}-\eta_0-Vt$,
$\eta_0=0.2$, and $V=0.5$, the reference free boundary is
$\Gamma_{\mathrm{ref}}(y,t)=\sqrt{2}(\eta_0+Vt)-y$.
The corresponding reference solution is
}
\begin{equation}
\label{eq:exact_sol_formula}
u_{\mathrm{ref}}(\boldsymbol{x},t)
=
\begin{cases}
2,
& \eta(\boldsymbol{x},t)<0,\\[3pt]
T_0+\dfrac{6}{5}
\exp\!\left(-\dfrac{\eta(\boldsymbol{x},t)}{2}\right),
& \eta(\boldsymbol{x},t)\geq0.
\end{cases}
\end{equation}
{\color{reviewerone}
The initial and fixed-boundary data are prescribed from this reference
solution, with $\Gamma(y,0)=\sqrt{2}\eta_0-y$.
At the interface,
$u_{1,\mathrm{ref}}=2$, $u_{2,\mathrm{ref}}=2.2$,
$q_1=0$, $q_2=0.6$, so all three relations in
\eqref{eq:Stefan_resistance} are satisfied.
\par}

{\color{reviewertwo}
This example uses the same thermo-kinematic splitting principle as the
basic model, with a different assignment of the interface conditions.
For a fixed free-boundary trajectory $\Gamma^{(k-1)}$, the two
phase-wise solutions are reconstructed jointly from the bulk equations,
initial and fixed-boundary data, and the first two relations in
\eqref{eq:Stefan_resistance}.
Since the geometry and $V_n^{(k-1)}$ are fixed, these interface
conditions are affine in the solution coefficients and therefore give a
linear least-squares thermodynamic subproblem.

The modified Stefan relation is reserved for the kinematic update:
\begin{equation}
\label{eq:thermal_stefan_velocity_target}
\left[
\beta+c\left(u_2^{(k)}-u_1^{(k)}\right)
\right]
V_n^{\mathrm{lag}}
\bigl(\widehat\Gamma;\Gamma^{(k-1)}\bigr)
=
q_2^{(k)}-q_1^{(k)},
\qquad
\text{on }\Gamma^{(k-1)} .
\end{equation}
With the reconstructed solutions $u_1^{(k)}$ and $u_2^{(k)}$ and the
current free-boundary trajectory $\Gamma^{(k-1)}$ fixed, the solution
values and fluxes entering this relation are fixed.
The updated free-boundary trajectory then enters through its normal
velocity, which is affine in the free-boundary coefficients under the
lagged formulation. Hence, the kinematic update is again a linear
least-squares problem with respect to the free-boundary coefficients.
The resulting trajectory is relaxed with $\rho=0.02$.
\par}
\begin{figure}[htbp]
    \centering
    \includegraphics[width=0.95\linewidth]
    {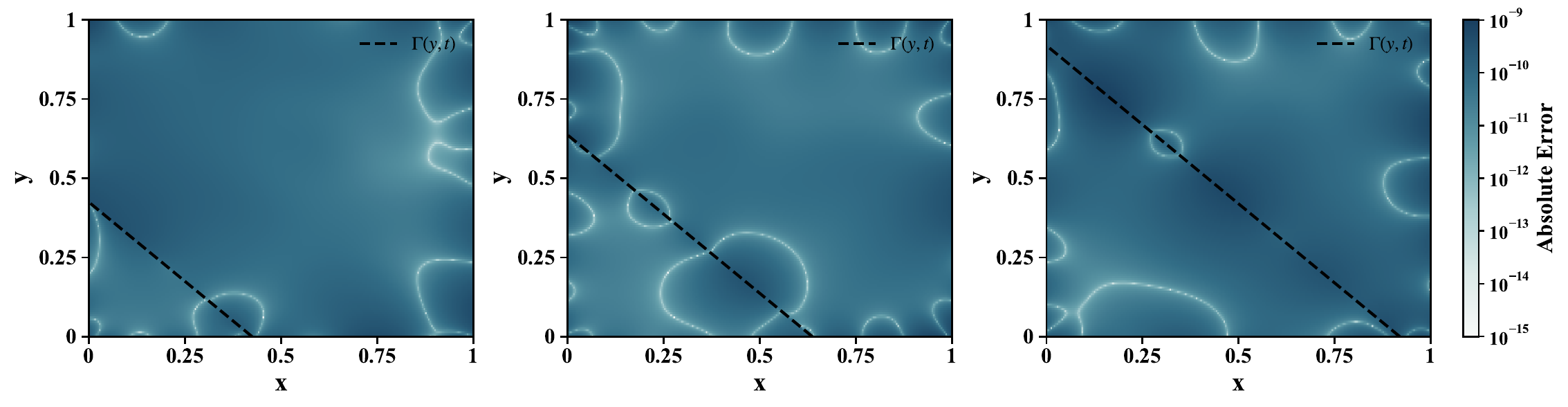}
    \caption{Solution error at $t=0.20$, $0.50$, and $0.90$
for the problem in Section~\ref{sec:thermal}.}
    \label{fig:exp7_Fig1_ErrorSnapshots}
\end{figure}
\begin{figure}[htbp]
    \centering
    \includegraphics[width=0.95\linewidth]
    {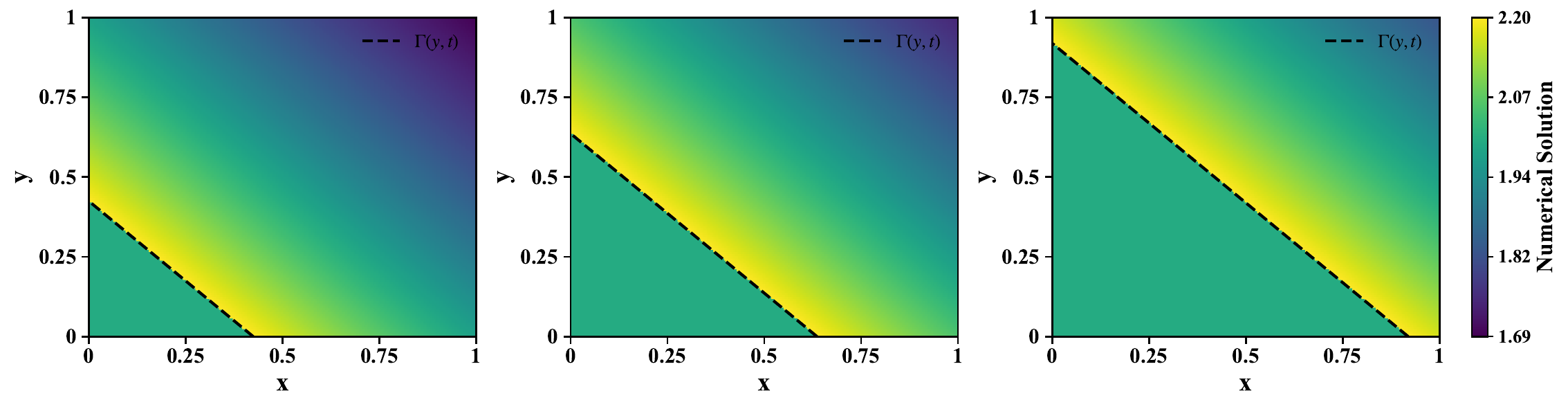}
    \caption{Numerical solutions at $t=0.20$, $0.50$, and $0.90$
for the problem in Section~\ref{sec:thermal}.}
    \label{fig:exp7_Fig5_NumericalSolutionHeatmaps}
\end{figure}
Figure~\ref{fig:exp7_Fig1_ErrorSnapshots} shows pointwise solution
errors of order $10^{-10}$.
Figure~\ref{fig:exp7_Fig5_NumericalSolutionHeatmaps} displays the
corresponding numerical solutions and the interfacial solution jump.
Figure~\ref{fig:exp7_global_errors}(a) shows the fixed-time relative
$L^2$ solution error, which remains between approximately
$3\times10^{-11}$ and $10^{-10}$ over $t\in[0,1]$.
Figure~\ref{fig:exp7_global_errors}(b) shows the reconstructed complete
free-boundary trajectory, colored by the pointwise absolute
free-boundary error.

The relative $L^2$ solution error is
$E_u=6.40\times10^{-11}$.
Because the solution is discontinuous across the interface, the
solution-error test points are partitioned by the reference free
boundary $\Gamma_{\mathrm{ref}}$.
At each test point, the numerical approximation associated with the
corresponding reference phase is compared with the analytical
solution in that phase.
The pointwise absolute free-boundary error remains below
$3.73\times10^{-11}$.

This benchmark tests the thermo-kinematic splitting with coupled
interface conditions that permit a temperature jump.

\begin{figure}[htbp]
    \centering

    \begin{subfigure}[b]{0.48\textwidth}
        \centering
        \includegraphics[width=\textwidth]
        {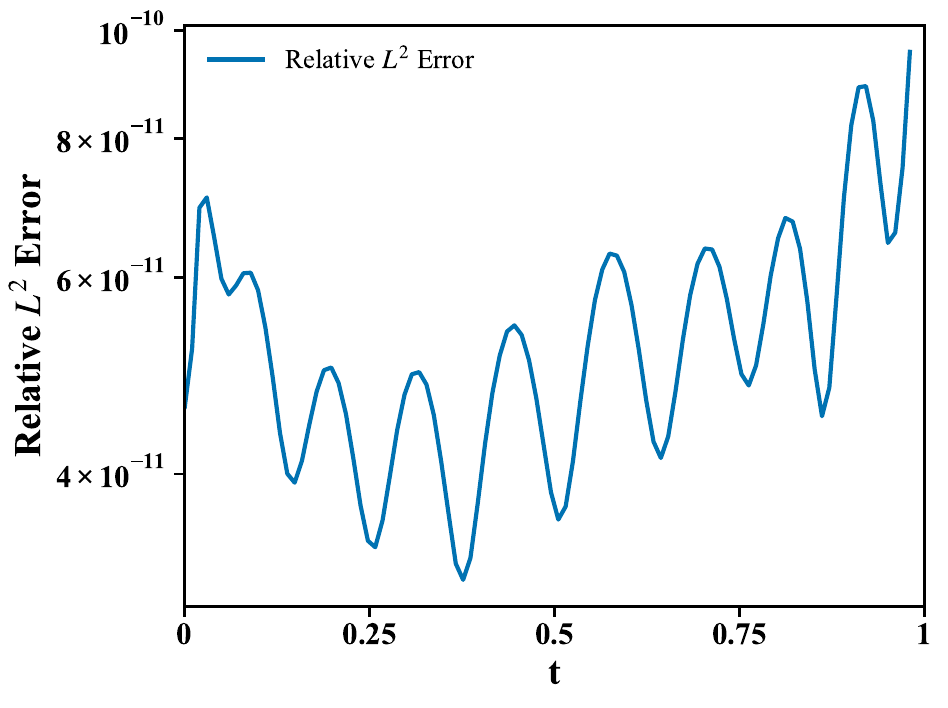}
        \caption{Fixed-time relative $L^2$ solution error.}
        \label{fig:exp7_Fig2_L2Error}
    \end{subfigure}
    \hfill
    \begin{subfigure}[b]{0.48\textwidth}
        \centering
        \includegraphics[width=\textwidth]
        {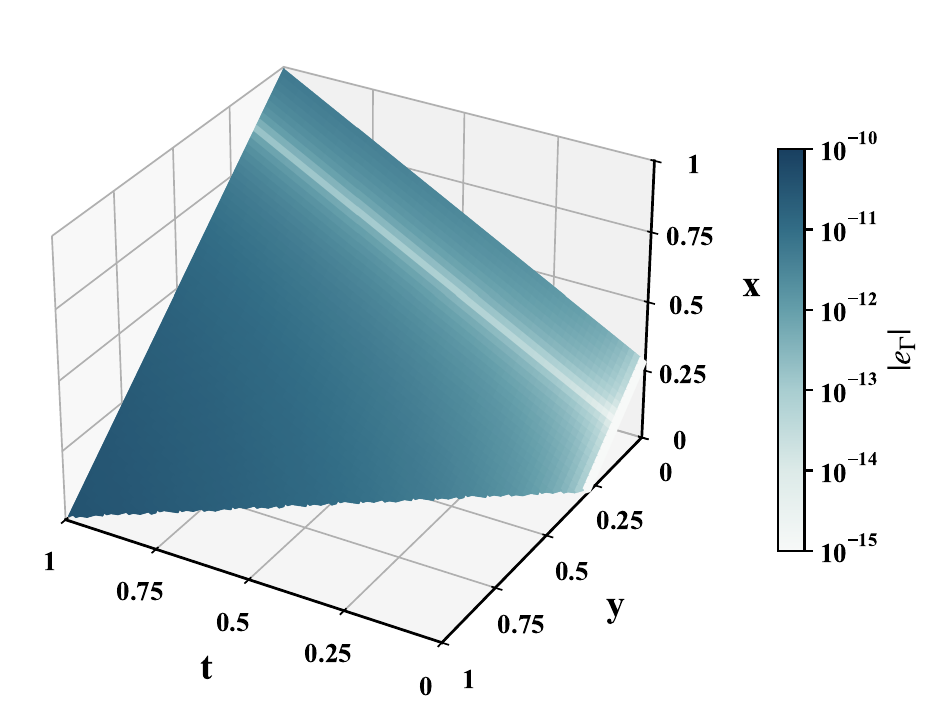}
        \caption{Predicted free-boundary trajectory colored by
        $|e_\Gamma(y,t)|$.}
        \label{fig:exp7_Fig3_Interface3D}
    \end{subfigure}

    \caption{Solution and free-boundary errors for the problem in
Section~\ref{sec:thermal}.}
    \label{fig:exp7_global_errors}
\end{figure}

\subsection{Frank sphere problems}
\label{sec:frank}

{\color{reviewerone}
The classical Frank similarity solutions are defined in unbounded
exterior liquid regions.
For the numerical computations below, these problems are truncated to
bounded domains, and the corresponding reference traces are imposed on
the artificial outer boundaries.
The corresponding restricted similarity solutions therefore remain
exact solutions of the bounded problems considered here.
In both tests, the prescribed initial interface is incorporated
exactly at $t=1$ by using the shifted time coordinate $t-1$
in~\eqref{eq:interface_solution_parameterization}.
\par}
\subsubsection*{Case 1. Two-dimensional Frank sphere}

The unbounded liquid region is truncated to
$\Omega=[-2,2]^2$, with
$\Omega_l(t)=\{(x,y)\in\Omega:r>R(\theta,t)\}$.
The solution satisfies
$u_t-\Delta u=0$ in $\Omega_l(t)$,
$u=u_m$ on $\Gamma(t)$, and
$\beta V_n=-\nabla u\cdot\boldsymbol n$ on $\Gamma(t)$, where
$\boldsymbol n$ points from the solid into the liquid.

We set $u_m=0$, $u_\infty=-1$, and $\lambda=0.5$.
With
$E_1(z)=\int_z^\infty e^{-s}s^{-1}\,{\rm d}s$,
the latent-heat parameter is
$\beta=(u_m-u_\infty)/
[\lambda^2e^{\lambda^2}E_1(\lambda^2)]$.
The exact solution and free-boundary radius are
\begin{equation}
\label{eq:frank2d_reference}
u_{\mathrm{ref}}(r,t)
=
u_\infty
+
(u_m-u_\infty)
\frac{E_1\!\left(r^2/(4t)\right)}
{E_1(\lambda^2)},
\qquad
R_{\mathrm{ref}}(t)=2\lambda\sqrt{t},
\qquad
t\in[1,3].
\end{equation}
The initial and outer-boundary data are prescribed from the reference
solution, with
$R(\theta,1)=1$.

Figure~\ref{fig:frank2D_error_snapshots} shows pointwise solution
errors of order $10^{-12}$.
Figure~\ref{fig:frank2D_errors}(a) shows the fixed-time relative
$L^2$ solution error $E_u(t)$, which remains below $10^{-11}$ over
$t\in[1,3]$.
Figure~\ref{fig:frank2D_errors}(b) displays the reconstructed complete
free-boundary trajectory, colored by the pointwise absolute
free-boundary error.
The relative $L^2$ free-boundary error is $E_\Gamma=2.58\times10^{-13}$.

\subsubsection*{Case 2. Three-dimensional Frank sphere}

We next consider the three-dimensional Frank benchmark, describing a
spherical solid growing into a supercooled liquid.
With $r=\|\boldsymbol{x}\|_2$, the solid is again maintained at the
melting temperature and the solution is computed in the liquid region
outside the moving spherical free boundary.

\begin{figure}[htbp]
\centering
\includegraphics[width=0.95\linewidth]
{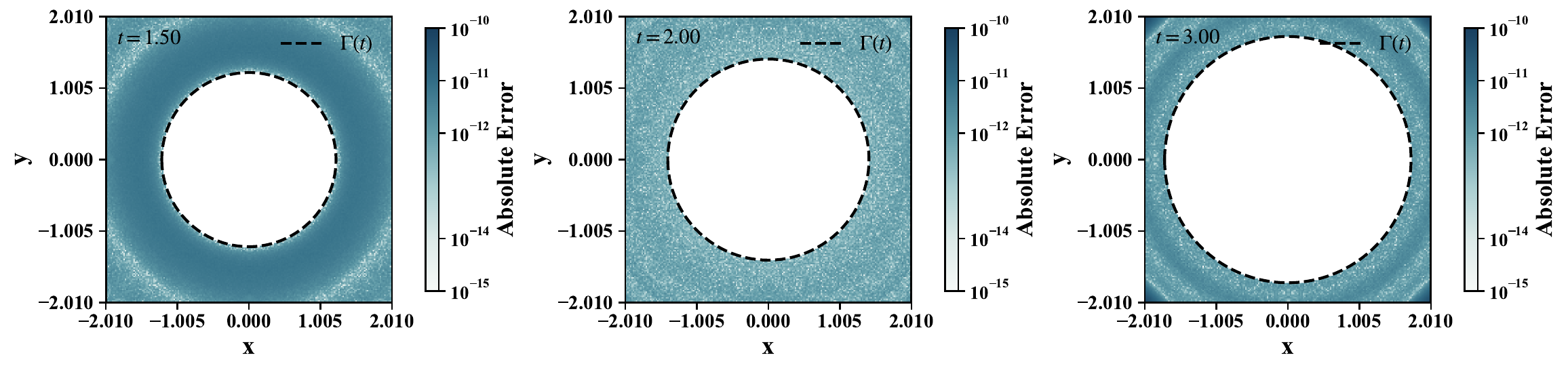}
\caption{Solution error at $t=1.50$, $2.00$, and $3.00$
for Case~1 in Section~\ref{sec:frank}.}
\label{fig:frank2D_error_snapshots}
\end{figure}

\begin{figure}[htbp]
\centering

\begin{subfigure}[b]{0.32\textwidth}
\centering
\includegraphics[width=\textwidth]
{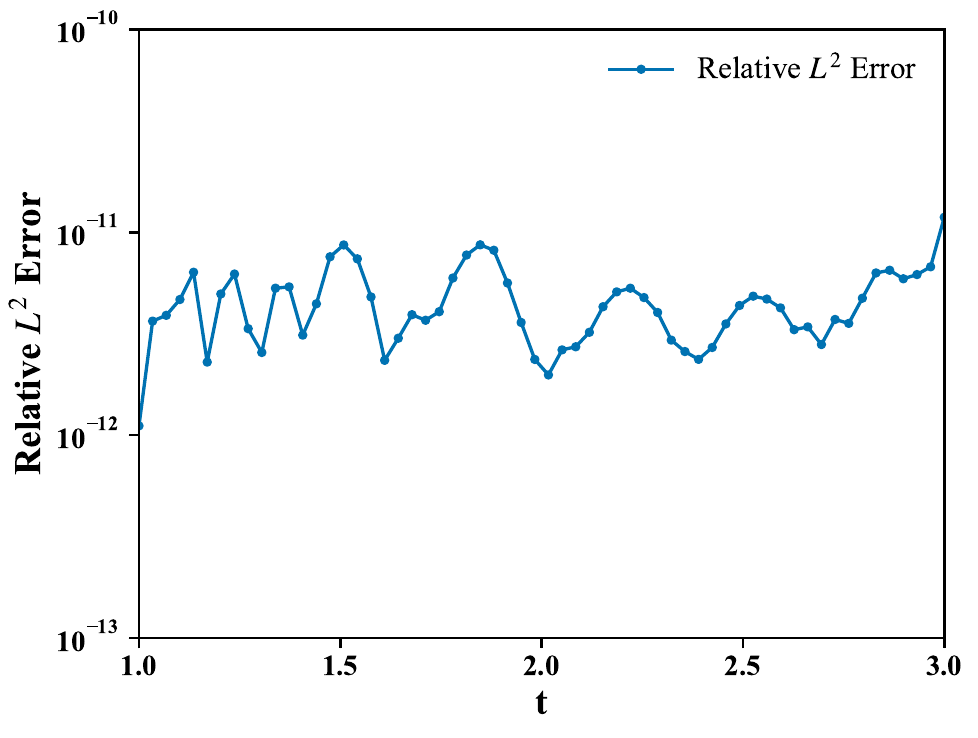}
\caption{Fixed-time relative $L^2$ solution error $E_u(t)$ for Case~1.}
\label{fig:frank2D_l2_error}
\end{subfigure}
\hfill
\begin{subfigure}[b]{0.32\textwidth}
\centering
\includegraphics[width=\textwidth]
{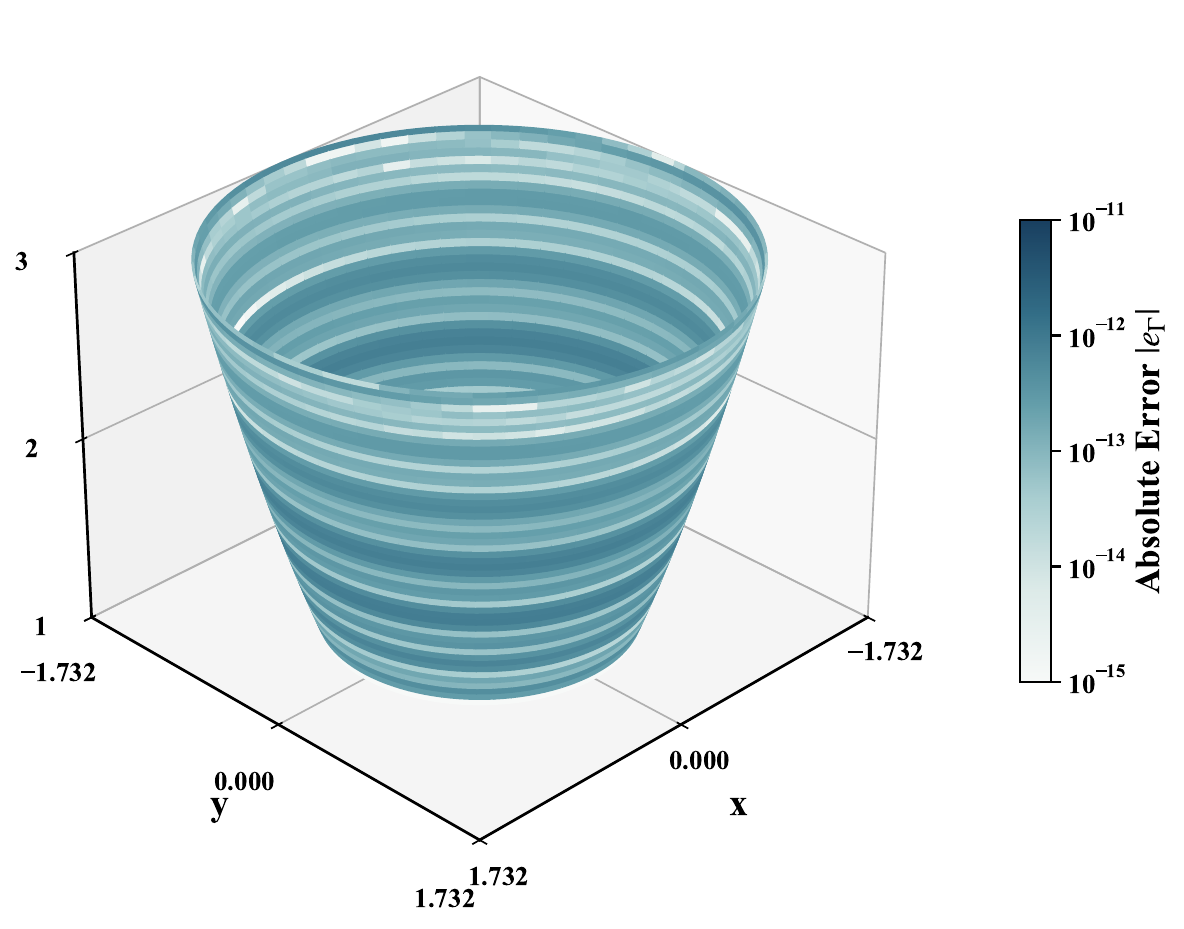}
\caption{Free-boundary trajectory for Case~1.}
\label{fig:frank2D_interface}
\end{subfigure}
\hfill
\begin{subfigure}[b]{0.32\textwidth}
\centering
\includegraphics[width=\textwidth]
{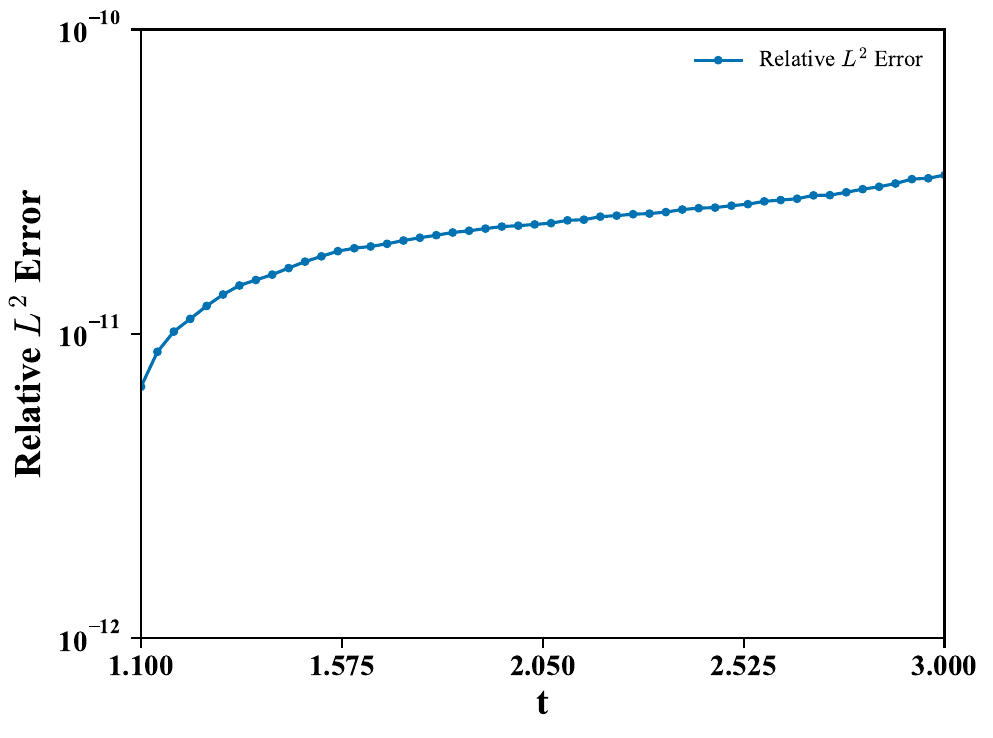}
\caption{Fixed-time relative $L^2$ solution error $E_u(t)$ for Case~2.}
\label{fig:frank3D_l2_error}
\end{subfigure}

\caption{Solution and free-boundary results for the Frank benchmarks in
Section~\ref{sec:frank}.}
\label{fig:frank2D_errors}
\end{figure}

The computational domain is
$\Omega=B_{3.5}(\boldsymbol 0)$, with
$\Omega_l(t)=\{\boldsymbol{x}\in\Omega:
r>R(\theta,\varphi,t)\}$.
The solution satisfies
$u_t-\Delta u=0$ in $\Omega_l(t)$,
$u=u_m$ on $\Gamma(t)$, and
$\beta V_n=-\nabla u\cdot\boldsymbol n$ on $\Gamma(t)$.

We set $u_m=0$, $u_\infty=-1$, and $\lambda=0.6$.
Define
$F(\eta)=e^{-\eta^2}/\eta-\sqrt{\pi}\operatorname{erfc}(\eta)$ and
$A=(u_m-u_\infty)/F(\lambda)$.
The exact solution and free-boundary radius are
\begin{equation}
\label{eq:frank3d_reference}
u_{\mathrm{ref}}(r,t)
=
u_\infty
+
A F\!\left(\frac{r}{2\sqrt{t}}\right),
\qquad
R_{\mathrm{ref}}(t)=2\lambda\sqrt{t},
\qquad
t\in[1,3].
\end{equation}
Since $F'(\eta)=-e^{-\eta^2}/\eta^2$,
$\beta
=(u_m-u_\infty)e^{-\lambda^2}/
[2\lambda^3F(\lambda)]$.
The initial and outer-boundary data are prescribed from the reference
solution, with $R(\theta,\varphi,1)=1.2$.

Figure~\ref{fig:frank3D_error_slices} shows pointwise solution
errors of order $10^{-11}$ on the planes $x=0$ and $y=0$.
Figure~\ref{fig:frank3D_interface_shells} shows the reconstructed
spherical free boundaries and their pointwise geometric errors.
The relative $L^2$ free-boundary error is
$E_\Gamma=2.25\times10^{-11}$.
Figure~\ref{fig:frank3D_l2_error} shows the fixed-time relative
$L^2$ solution error $E_u(t)$ over $t\in[1,3]$.
\begin{figure}[htbp]
\centering
\includegraphics[width=0.8\linewidth]
{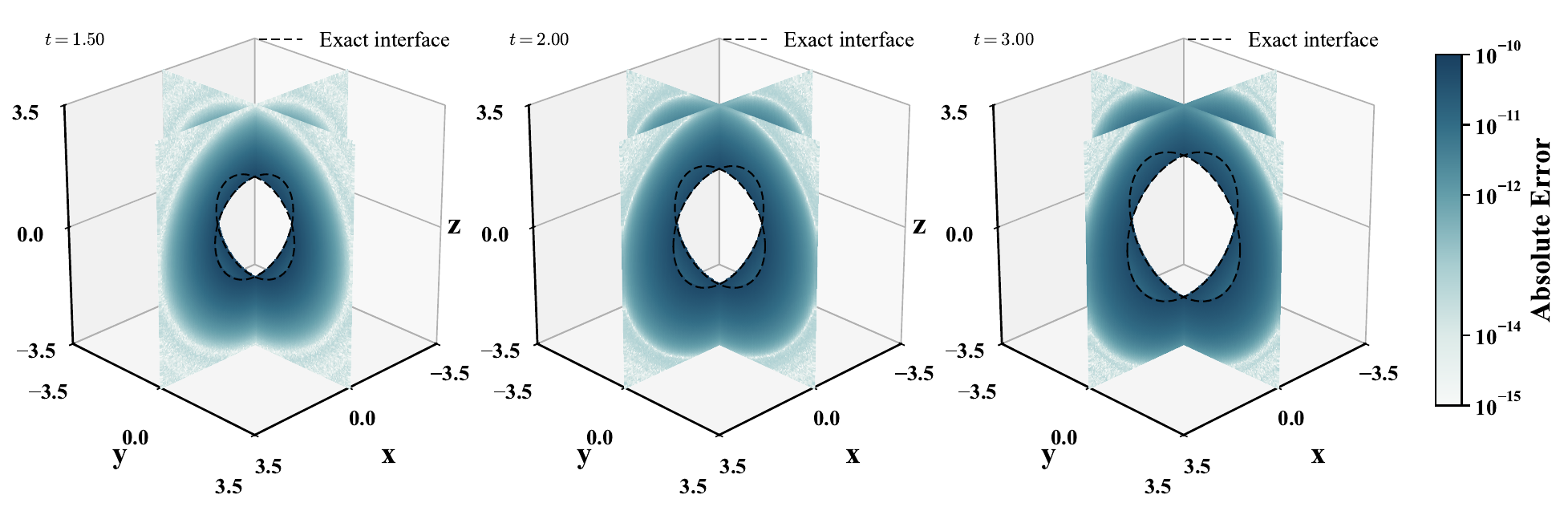}
\caption{Solution error on the planes $x=0$ and $y=0$ at
$t=1.50$, $2.00$, and $3.00$ for Case~2 in
Section~\ref{sec:frank}.}
\label{fig:frank3D_error_slices}
\end{figure}
\begin{figure}[htbp]
\centering
\includegraphics[width=0.8\linewidth]
{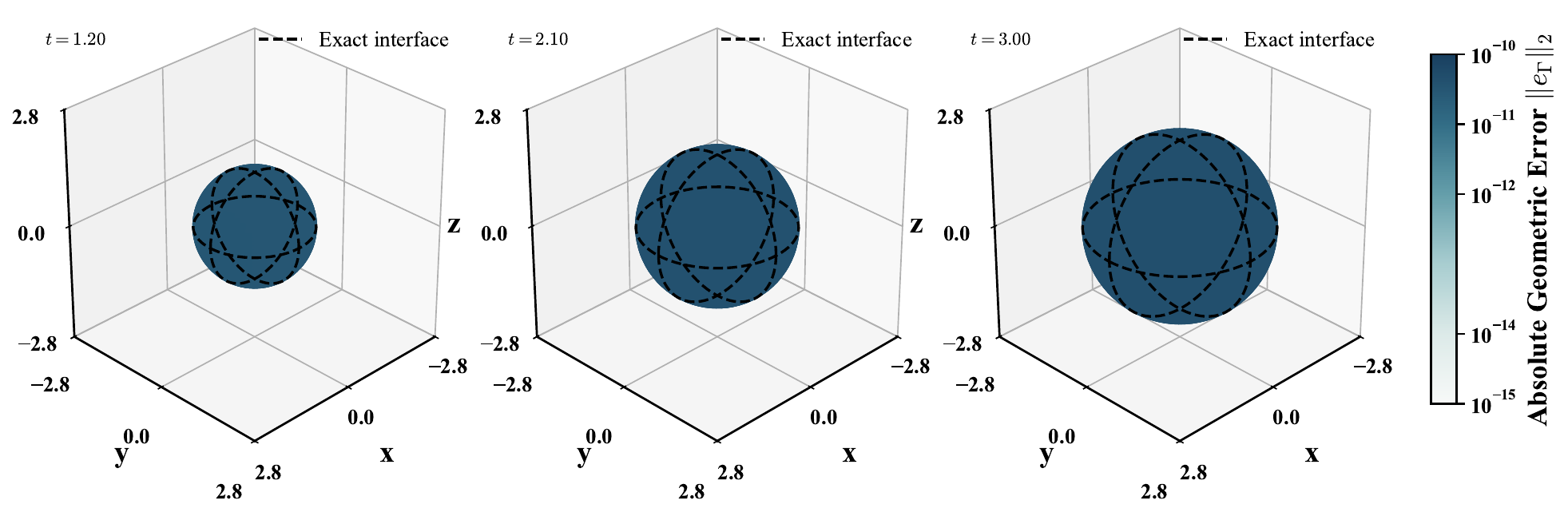}
\caption{Free-boundary surfaces colored by
$\|e_\Gamma(\theta,\varphi,t)\|_2$ at
$t=1.20$, $2.10$, and $3.00$ for Case~2 in
Section~\ref{sec:frank}.}
\label{fig:frank3D_interface_shells}
\end{figure}

\subsection{A quasi-stationary Mullins--Sekerka problem}
\label{sec:mullins}

We consider a two-dimensional one-sided quasi-stationary Stefan
problem with a Gibbs--Thomson condition.
This problem belongs to the class of curvature-dependent free-boundary
flows discussed in Ref.~\cite{eto2024rapid}.
Here, the liquid occupies a bounded region outside the solid, with a
Dirichlet condition prescribed on a fixed outer circle.

Let $\Gamma(t)$ denote the solid--liquid free boundary and
$\Omega_{\ell}(t)$ the liquid region between $\Gamma(t)$ and the
outer circle $r=R_{\mathrm{out}}$.
The dimensionless problem is
\begin{equation}
\label{eq:ms_finite_problem}
\begin{aligned}
\Delta u&=0
&&\text{in }\Omega_{\ell}(t),\\
u&=u_{\mathrm{out}}
&&\text{on }r=R_{\mathrm{out}},\\
u&=-\gamma\kappa
&&\text{on }\Gamma(t),\\
V_n&=-\nabla u\cdot\boldsymbol n
&&\text{on }\Gamma(t),
\end{aligned}
\end{equation}
where $\boldsymbol n$ points from the solid into the liquid and
$\kappa=\nabla_\Gamma\cdot\boldsymbol n$ denotes the signed
interface curvature, so that $\kappa=1/R$ for a circular interface
of radius $R$.

We set $R_{\mathrm{out}}=4$, $u_{\mathrm{out}}=1$, and
$0\leq t\leq0.25$, with
$R(\theta,0)=0.5+0.15\cos(5\theta)$.
We consider $\gamma=0.01$, $0.005$, and $0.0001$.
For this illustrative example, we use a model-specific realization of
the thermo-kinematic splitting principle, in which the solution
approximation is constructed from fundamental solutions of the Laplace
equation.
\textcolor{reviewerone}{
Each outer iteration updates the complete free-boundary trajectory
over the full time interval, without time stepping.}
An independent boundary-integral computation for the same
finite-container problem is used as the numerical reference.

\begin{figure}[htbp]
    \centering
    \includegraphics[width=0.8\linewidth]
    {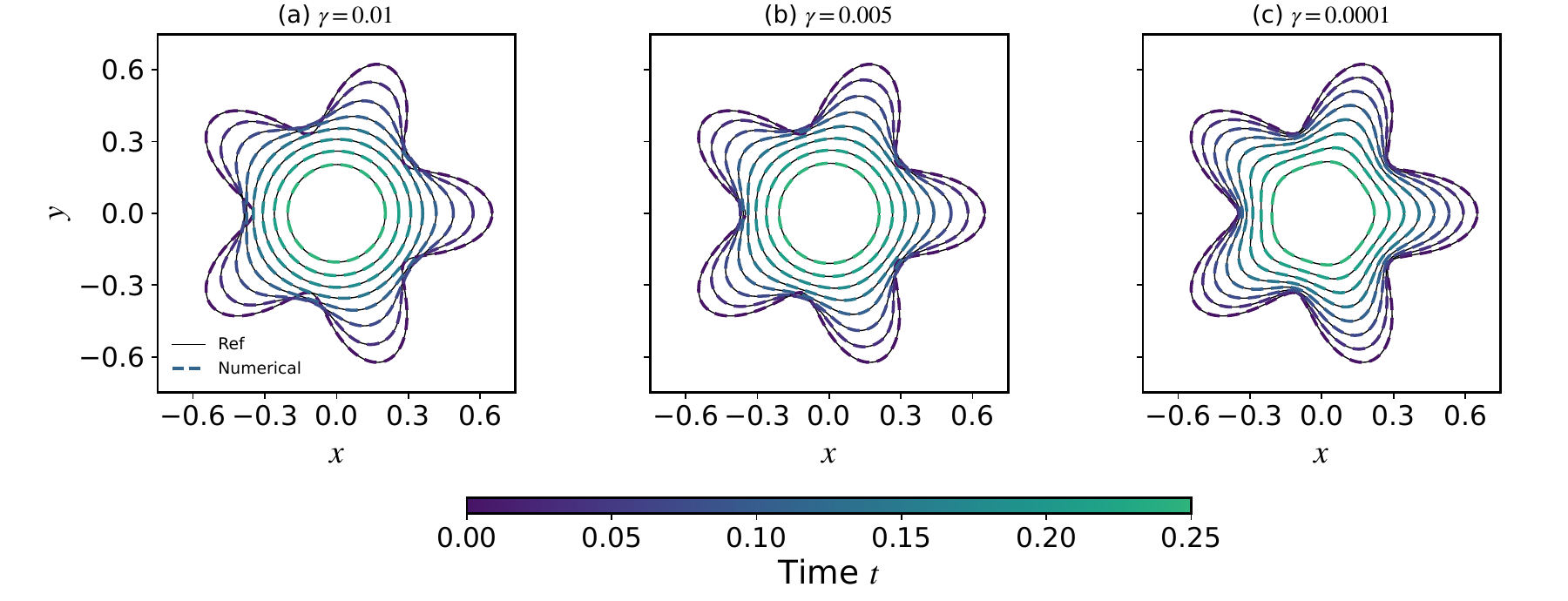}
    \caption{Free-boundary evolution for $\gamma=0.01$, $0.005$, and
$0.0001$ in Section~\ref{sec:mullins}.}
    \label{fig:Gibbs}
\end{figure}

Figure~\ref{fig:Gibbs} compares the reconstructed interfaces
(dashed curves) with the independent numerical reference
(solid curves).
The two results show close agreement over the reported interval.
The interfaces shrink, while larger values of $\gamma$ suppress
the initial angular perturbation more rapidly.
{\color{reviewerone}
\subsection{Three-dimensional spherical melting--vaporization problem
with two moving free boundaries}
\label{sec:two_interface_spherical}

We next consider a dimensionless spherical melting--vaporization
problem with two moving free boundaries, based on the Stefan model of
Kavokin et al.~\cite{kavokin2016exact}, over
$t\in[t_0,T]=[1,2]$.
With $r=\|\boldsymbol{x}\|_2$, the inner free boundary
$\Gamma_1(t)$ represents the vaporization surface and
$\Gamma_2(t)$ separates the molten region from the exterior solid.
Let $D=[-3.6,3.6]^3$. The computational regions are
$\Omega_1(t)=\{\boldsymbol{x}\in D:r>R_2(t)\}$
for the solid and
$\Omega_2(t)=\{\boldsymbol{x}\in D:R_1(t)<r<R_2(t)\}$
for the liquid. The phase-wise solutions satisfy
$u_{i,t}-k_i\Delta u_i=0$ in $\Omega_i(t)$, $i=1,2$.
Let $\boldsymbol m_j$ denote the outward unit normal to the region
enclosed by $\Gamma_j$, $j=1,2$, and let $\beta>0$ be the
dimensionless latent-heat parameter. The inner free boundary satisfies
\begin{equation}
\label{eq:aip_inner_conditions}
u_2=u_b,
\qquad
-\alpha_2\nabla u_2\cdot\boldsymbol m_1
=\frac{P(t)}{2\pi R_1(t)^2},
\qquad \text{on }\Gamma_1(t).
\end{equation}
At the outer free boundary, the normal from solid to liquid is
$\boldsymbol n_2=-\boldsymbol m_2$, as in
Section~\ref{sec:problem}, and
\begin{equation}
\label{eq:aip_outer_conditions}
\begin{aligned}
u_1&=u_2=u_m,
&&\text{on }\Gamma_2(t),\\
\beta V_{n,2}
&=(\alpha_1\nabla u_1-\alpha_2\nabla u_2)
\cdot\boldsymbol n_2,
&&\text{on }\Gamma_2(t).
\end{aligned}
\end{equation}
For spherical interfaces, $\boldsymbol m_j=\boldsymbol e_r$
and $V_{n,2}=-R_2'(t)$, so the outer Stefan condition is
$\beta R_2'=\alpha_1u_{1,r}-\alpha_2u_{2,r}$.

For the reference solution, define
$H(\eta)=\operatorname{erfc}(\eta)
-\exp(-\eta^2)/(\sqrt{\pi}\eta)$.
The reference free-boundary trajectories are
$R_{1,\mathrm{ref}}(t)=2a\sqrt{t}$ and
$R_{2,\mathrm{ref}}(t)=2b\sqrt{t}$, while the corresponding
phase-wise solutions are
\begin{align}
\label{eq:aip_reference_u1}
u_{1,\mathrm{ref}}(r,t)
&=u_m
\frac{H\!\left(r/(2\sqrt{k_1t})\right)}{H(b/\sqrt{k_1})},\\
\label{eq:aip_reference_u2}
u_{2,\mathrm{ref}}(r,t)
&=u_b+
\frac{u_m-u_b}{H(b/\sqrt{k_2})-H(a/\sqrt{k_2})}
\left[
H\!\left(\frac{r}{2\sqrt{k_2t}}\right)-H(a/\sqrt{k_2})
\right].
\end{align}
The prescribed input $P(t)$ is obtained from the reference flux on
$\Gamma_1(t)$, and $\beta$ is chosen so that the reference solution
satisfies the outer Stefan condition.

We use
$k_1=1.44$, $k_2=1$, $\alpha_1=0.05$, $\alpha_2=1$,
$u_b=1$, $u_m=0.4$, $a=0.4$, $b=1$,
$t_0=1$, and $T=2$, giving
$\beta\approx0.07336$.
The initial data and interface locations are prescribed
from the reference solution at $t=t_0$.
On the six faces of $D$, we prescribe the Dirichlet data
$u_1(\boldsymbol{x},t)=u_{1,\mathrm{ref}}(\|\boldsymbol{x}\|_2,t)$.
The restricted reference solution therefore remains an exact
solution of the bounded problem and is used to evaluate the
reported errors.
We also test the far-field approximation obtained by imposing
$u_1=0$ at $r=R_\infty=10$; the results are reported in \ref{app:homogeneous_boundary}.

Figure~\ref{fig:aip_error_sections} shows the pointwise solution
errors on the planes $x=0$ and $y=0$.

With the current free-boundary trajectories fixed, the solution
reconstruction uses the heat equations, the initial and
outer-boundary data, and the prescribed interface values.
The inner flux relation and the outer Stefan condition determine
the subsequent geometry update.
Let $R_j^M(\xi,t;\mathbf c_j)$ be radial representations that are
affine in $\mathbf c_j$ and incorporate the prescribed initial
radii exactly at $t=t_0$, where $\xi$ denotes the spherical angles.
We form a radius target for the inner free boundary and a velocity
target for the outer free boundary:
\begin{equation}
\label{eq:aip_fixed_targets}
\begin{aligned}
q^{(k)}
&=-\alpha_2\nabla u_2^{(k)}
  \cdot\boldsymbol m_1^{(k-1)},\\
T_1^{(k)}
&=\left(\frac{P(t)}{2\pi q^{(k)}}\right)^{1/2},\\
T_2^{(k)}
&=\frac{
(\alpha_1\nabla u_1^{(k)}-\alpha_2\nabla u_2^{(k)})
\cdot\boldsymbol m_2^{(k-1)}
}{
\beta\,\boldsymbol e_r\cdot\boldsymbol m_2^{(k-1)}
}.
\end{aligned}
\end{equation}
All solution traces and normals are evaluated on the current
free-boundary trajectories.
The inner-radius target requires $q^{(k)}>0$ at the fitting points.
At the interface parameter--time points $(\xi_\ell,t_\ell)$,
the candidate coefficients are determined by
\begin{equation}
\label{eq:aip_interface_fits}
\begin{aligned}
\widetilde{\mathbf c}_1^{(k)}
&=\operatorname*{argmin}_{\mathbf c_1}
\sum_\ell
\left|R_1^M(\xi_\ell,t_\ell;\mathbf c_1)
-T_1^{(k)}(\xi_\ell,t_\ell)\right|^2,\\
\widetilde{\mathbf c}_2^{(k)}
&=\operatorname*{argmin}_{\mathbf c_2}
\sum_\ell
\left|\partial_tR_2^M(\xi_\ell,t_\ell;\mathbf c_2)
-T_2^{(k)}(\xi_\ell,t_\ell)\right|^2,
\end{aligned}
\end{equation}
followed by coefficient relaxation.
Both targets are fixed during these convex linear least-squares fits.

\begin{figure}[htbp]
\centering
\includegraphics[width=0.8\linewidth]
{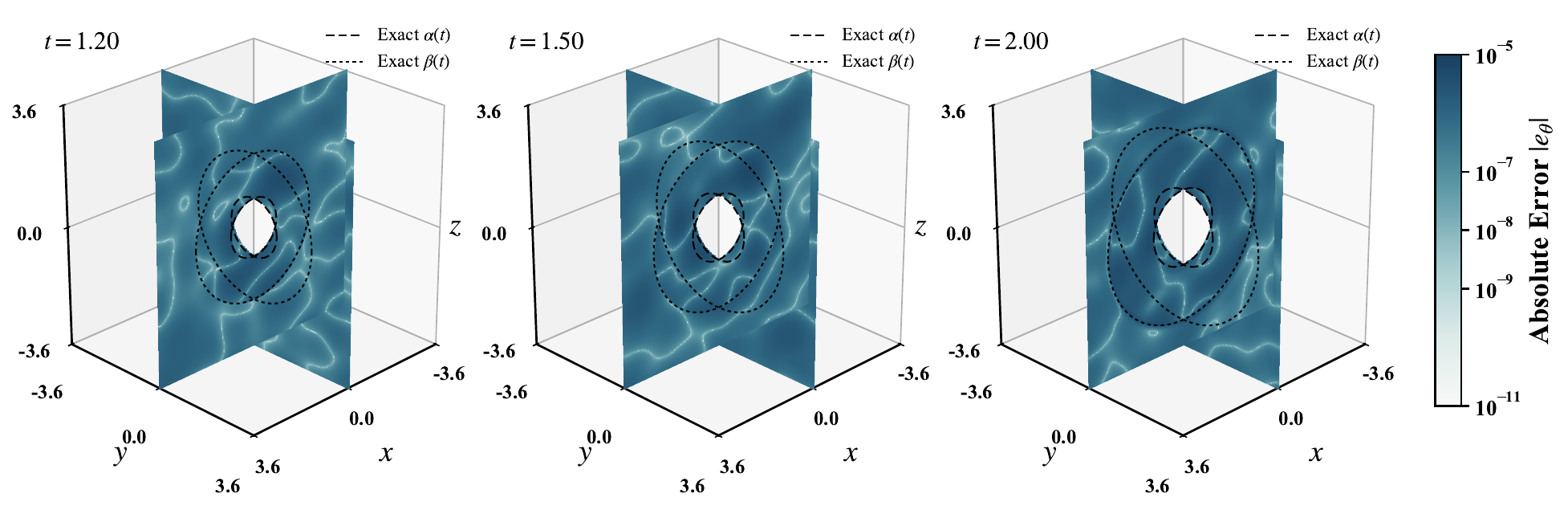}
\caption{Solution error on the planes $x=0$ and $y=0$ at
$t=1.20$, $1.50$, and $2.00$ for the problem in
Section~\ref{sec:two_interface_spherical}.}
\label{fig:aip_error_sections}
\end{figure}
\begin{figure}[htbp]
\centering
\includegraphics[width=0.8\linewidth]
{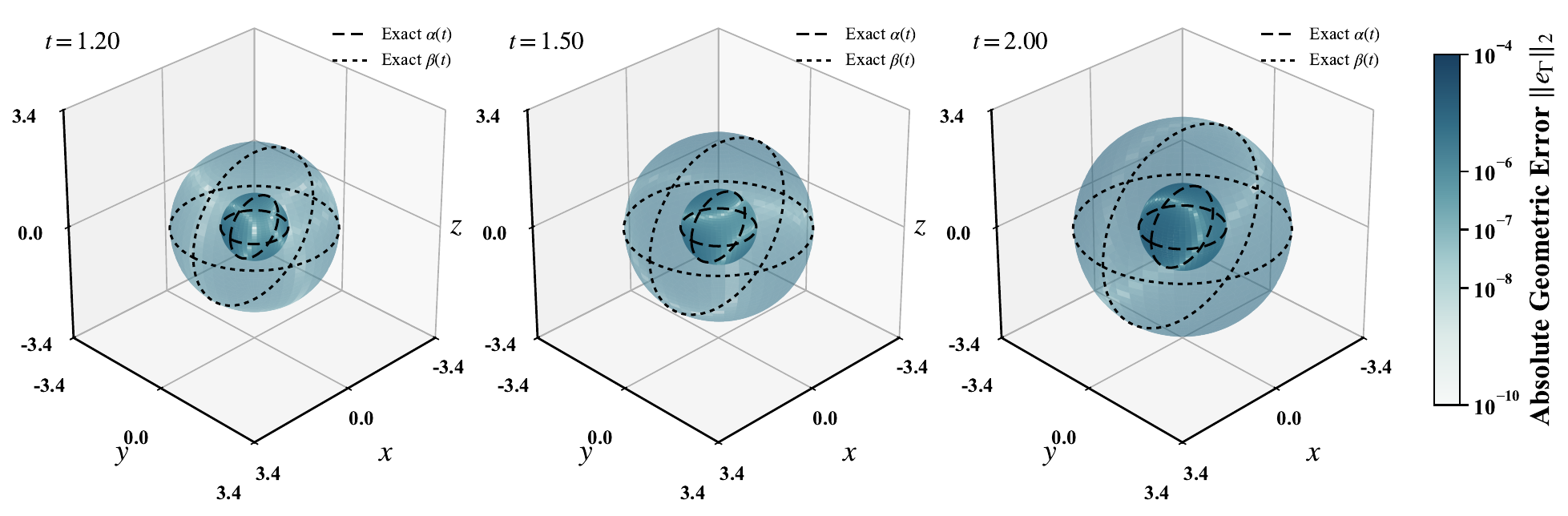}
\caption{Inner and outer free-boundary surfaces and pointwise errors at
$t=1.20$, $1.50$, and $2.00$ for the problem in
Section~\ref{sec:two_interface_spherical}.}
\label{fig:aip_double_interface_shells}
\end{figure}
Figure~\ref{fig:aip_double_interface_shells} shows the reconstructed
inner and outer free boundaries, colored by the corresponding
pointwise free-boundary errors.
The relative $L^2$ free-boundary errors are
$E_{\Gamma_1}=3.30\times10^{-6}$ and
$E_{\Gamma_2}=2.69\times10^{-6}$, while the phase-wise space--time
relative solution errors are
$E_{u_1}=2.79\times10^{-6}$ and
$E_{u_2}=5.53\times10^{-6}$.
The fixed-time relative solution errors are shown in
Figure~\ref{fig:aip_l2_evolution}.

Unlike the single-interface Frank benchmark, this problem requires the
simultaneous reconstruction of two free-boundary trajectories coupled
through the solution in the intervening molten region and governed by
different interface conditions.

\subsection{Three-dimensional spherical solidification with a Gibbs--Thomson condition}
\label{sec:case_tausch_spherical}

We finally consider spherical solidification in an undercooled liquid,
following Case and Tausch~\cite{case2012integral}.
With $r=\|\boldsymbol{x}\|_2$,
$\Omega_1(t)=\{\boldsymbol{x}:r<R(t)\}$ is the solid region and
$\Omega_2(t)=\{\boldsymbol{x}:R(t)<r<R_\infty\}$ the liquid region.
The phase-wise solutions satisfy
$u_{i,t}-k_i\Delta u_i=0$ in $\Omega_i(t)$, $i=1,2$.
At the moving free boundary,
\begin{equation}
\label{eq:case_tausch_gibbs_thomson}
u_1(R(t),t)
=
u_2(R(t),t)
=
1-\frac{\gamma}{R(t)},
\end{equation}
and, with $\beta=1$,
\begin{equation}
\label{eq:case_tausch_stefan}
\frac{\mathrm dR}{\mathrm dt}
=
\alpha_1
\frac{\partial u_1}{\partial r}(R(t),t)
-
\alpha_2
\frac{\partial u_2}{\partial r}(R(t),t).
\end{equation}
The initial data are
$u_1(r,0)=1-\gamma/R_0$,
$u_2(r,0)=(R_0-\gamma)/r$, and $R(0)=R_0$.
We use
$k_1=60.31$, $k_2=40.82$, $\alpha_1=\alpha_2=1$,
$\gamma=0.01$, $R_0=0.1$, and $0\leq t\leq1$.

For a prescribed free-boundary trajectory, the Gibbs--Thomson value
$1-\gamma/R(t)$ is known interface data for the phase-wise heat
equations, while the resulting interfacial heat flux determines the
subsequent free-boundary update through
\eqref{eq:case_tausch_stefan}.
Thus, the radius-dependent interface value modifies the
thermodynamic boundary data without changing the thermo-kinematic
splitting structure.

For $\gamma=0.01$, the reference free-boundary trajectory is computed
using the integral-equation algorithm of Case and
Tausch~\cite{case2012integral}, whose liquid-phase problem is posed on
an unbounded exterior domain.
In the present computation, the liquid region is truncated at
\[
R_\infty=28,
\qquad
u_2(R_\infty,t)=0.
\]
The initial solution profile is retained without modification,
so its trace at $r=R_\infty$ is incompatible with the homogeneous
outer-boundary condition at $t=0$.
The resulting numerical trajectory is compared with the
Case--Tausch trajectory as a finite-domain approximation to the
unbounded-domain reference problem.

Although the reference problem is spherically symmetric and its
radius depends only on $t$, we allow angular dependence in the
numerical representation $R^M(\theta,\varphi,t)$ rather than
imposing spherical symmetry on the free boundary.
Figure~\ref{fig:case_tausch_interface_comparison} compares the angular
mean and angular range of the reconstructed radius with the
Case--Tausch reference trajectory.
The predicted mean radius follows the reference growth curve, while
the small angular range indicates that the reconstructed interface
remains close to the spherical configuration.

\begin{figure}[htbp]
    \centering

    \begin{subfigure}[b]{0.48\textwidth}
        \centering
        \includegraphics[width=\textwidth]
        {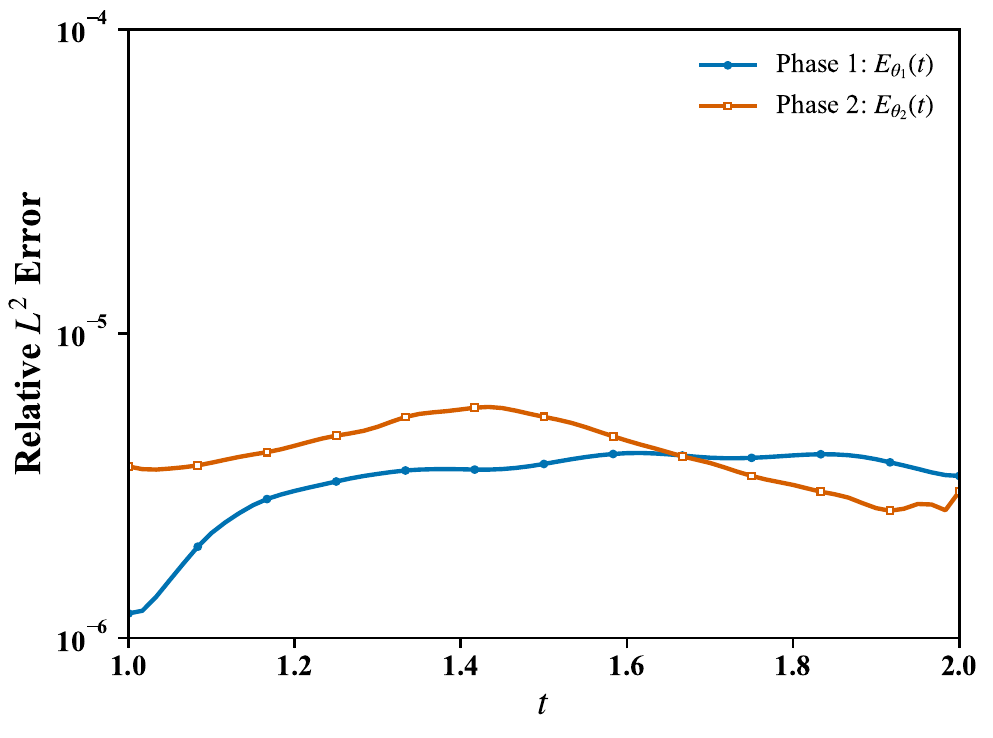}
        \caption{Fixed-time relative $L^2$ solution errors for
        numerical example~\ref{sec:two_interface_spherical}.}
        \label{fig:aip_l2_evolution}
    \end{subfigure}
    \hfill
    \begin{subfigure}[b]{0.48\textwidth}
        \centering
        \includegraphics[width=\textwidth]
        {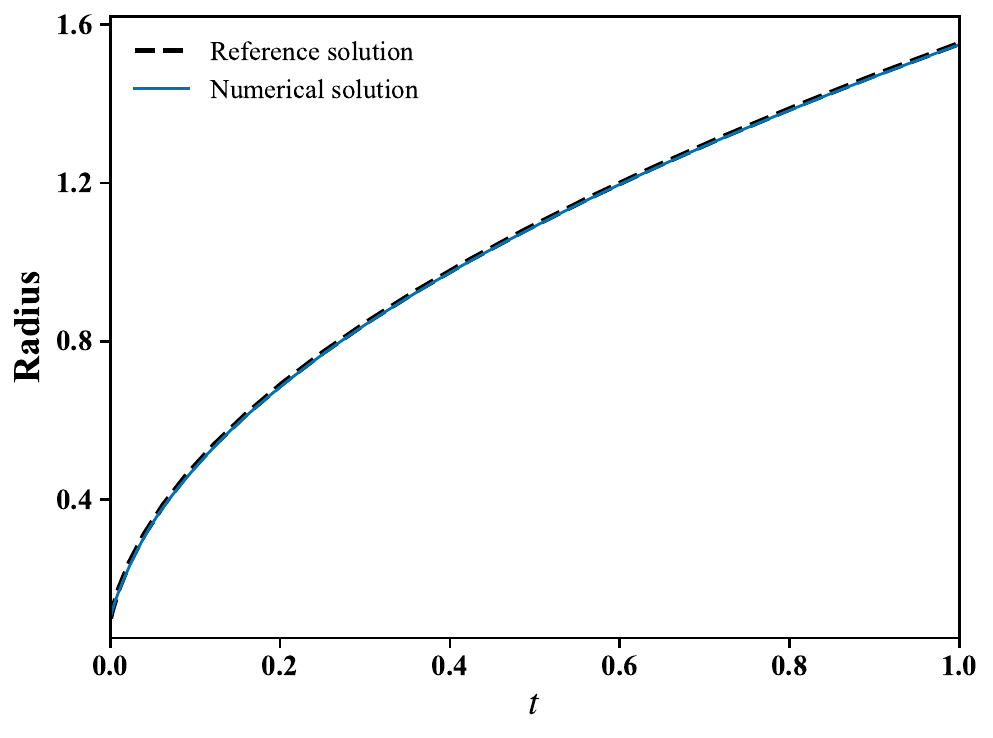}
        \caption{Comparison of the reconstructed and reference
        free-boundary trajectories for
        numerical example~\ref{sec:case_tausch_spherical}.}
        \label{fig:case_tausch_interface_comparison}
    \end{subfigure}

    \caption{Fixed-time solution errors for the two-interface spherical
problem and free-boundary comparison for the Gibbs--Thomson
solidification problem.}
    \label{fig:aip_g_spherical_results}
\end{figure}
\par}

\section{Conclusion}
\label{sec:conclusion}
{\color{reviewerone}
In this work, we proposed a thermo-kinematic operator splitting
framework based on Extreme Learning Machines for Stefan problems.
Both the phase-wise solutions and the complete free-boundary
trajectory are represented over the full space--time domain,
without time stepping.
The framework treats the solution--free-boundary coupling through
thermodynamic solution reconstruction on the current trajectory
and a geometry-lagged kinematic update of that trajectory.
For the models considered here, these two split steps are formulated
as linear least-squares problems and are therefore convex, while the
remaining nonlinear coupling is retained in a reduced fixed-point
iteration for the complete free-boundary trajectory.
ELMs provide finite-dimensional space--time trial representations
in which the unknown output coefficients enter linearly, allowing
the split subproblems to be solved directly by linear algebra.
\par}
{\color{reviewertwo}
For a fixed space--time discretization, the analysis establishes
local linear convergence of the relaxed interface iteration near
an assumed discrete fixed trajectory under the stated regularity
and spectral assumptions and with suitable relaxation.
For the basic Stefan model, the lagged and unlagged kinematic
residuals coincide at a split fixed point, so small split residuals
imply a small coupled discrete residual.
\par}

{\color{reviewerone}
Numerical experiments in one, two, and three spatial dimensions
evaluate the reconstruction of the solutions and free-boundary
trajectories against analytical or numerical references.
Across the four verification benchmarks, the proposed framework
yields smaller solution and free-boundary errors than the
monolithic ELM--GN baseline under matched discretizations.
The thermal-resistance and multiple-interface examples further
test the splitting under interface conditions involving temperature
jumps and coupled free boundaries.
The spherical Gibbs--Thomson example compares the reconstructed
radius with a numerical reference trajectory, while the
Mullins--Sekerka example provides a qualitative illustration of
curvature-dependent interface evolution.
\par}

The linear least-squares structure considered here relies on
thermodynamic and kinematic residuals that are affine in their
respective unknown coefficients once the required geometric
quantities and fitting targets have been fixed.
Extensions to nonlinear bulk equations, more general constitutive
relations, and rigorous discretization error analysis constitute
directions for future research.
\section*{Acknowledgments} 
	This work was funded by the National Natural Science Foundation of China (grant No. 12571431).
\newpage
\appendix
\section{Neuron-number study}
\label{app:case2_neuron_study}

\begin{table}[htbp]
\centering
\caption{Errors for different numbers of solution and free-boundary neurons
for Case~2 in Section~\ref{sec:1D2phase}.}
\label{tab:case2_neuron_study}
\small
\begin{tabular}{cc cccc}
\toprule
Solution neurons
& Free-boundary neurons
& $E_{u_1}$
& $E_{u_2}$
& $E_{\Gamma}$
& Stefan RMS \\
\midrule

\multirow{2}{*}{100}
& 100
& $2.925\times10^{-5}$
& $1.081\times10^{-5}$
& $5.719\times10^{-6}$
& $3.047\times10^{-12}$ \\
& 200
& $2.925\times10^{-5}$
& $1.081\times10^{-5}$
& $5.719\times10^{-6}$
& $4.541\times10^{-12}$ \\
\midrule

\multirow{2}{*}{200}
& 100
& $5.914\times10^{-10}$
& $3.211\times10^{-10}$
& $9.228\times10^{-11}$
& $4.854\times10^{-13}$ \\
& 200
& $5.914\times10^{-10}$
& $3.211\times10^{-10}$
& $9.225\times10^{-11}$
& $4.411\times10^{-13}$ \\
\midrule

\multirow{2}{*}{500}
& 100
& $9.028\times10^{-13}$
& $1.289\times10^{-12}$
& $2.963\times10^{-13}$
& $1.163\times10^{-12}$ \\
& 200
& $9.925\times10^{-13}$
& $1.392\times10^{-12}$
& $1.973\times10^{-13}$
& $1.875\times10^{-12}$ \\
\midrule

\multirow{2}{*}{1000}
& 100
& $6.518\times10^{-13}$
& $8.264\times10^{-13}$
& $9.247\times10^{-14}$
& $7.555\times10^{-13}$ \\
& 200
& $6.456\times10^{-13}$
& $7.940\times10^{-13}$
& $9.076\times10^{-14}$
& $1.022\times10^{-12}$ \\
\midrule

\multirow{2}{*}{1500}
& 100
& $4.305\times10^{-13}$
& $1.365\times10^{-12}$
& $6.893\times10^{-14}$
& $4.996\times10^{-13}$ \\
& 200
& $4.418\times10^{-13}$
& $1.329\times10^{-12}$
& $7.022\times10^{-14}$
& $5.515\times10^{-13}$ \\

\bottomrule
\end{tabular}
\end{table}

\section{Homogeneous outer boundary}
\label{app:homogeneous_boundary}

For the problem in Section~\ref{sec:two_interface_spherical},
we also truncate the exterior solid region at $r=R_\infty=10$
and impose $u_1=0$ there to approximate the far-field condition.
The physical parameters and time interval are unchanged.
Table~\ref{tab:homogeneous_boundary} reports the relative $L^2$
errors against the unbounded-domain reference solution.
These errors include both numerical approximation errors and
the effect of the approximate outer-boundary condition.

\begin{table}[htbp]
\centering
\caption{Relative errors with a homogeneous outer boundary.}
\label{tab:homogeneous_boundary}
\begin{tabular}{cccc}
\toprule
$E_{\Gamma_1}$ & $E_{\Gamma_2}$ & $E_{u_1}$ & $E_{u_2}$ \\
\midrule
$8.50\times10^{-5}$ & $6.71\times10^{-5}$ &
$8.98\times10^{-4}$ & $1.06\times10^{-4}$ \\
\bottomrule
\end{tabular}
\end{table}

\bibliographystyle{elsarticle-num-names} 
\bibliography{reference}






\end{document}